\documentclass[11pt]{article}
\usepackage{enumerate,comment}
\usepackage{tikz}
\usetikzlibrary{shapes}
\usepackage{hyperref}
\usepackage{microtype,a4wide}
\usepackage{boxedminipage}
\usepackage{xspace}
\hypersetup{
    pdftitle=   {Tree pivot-minors and linear rank-width},
   pdfauthor=  {Konrad K. Dabrowski, Fran\c{c}ois Dross, Jisu Jeong, Mamadou Moustapha Kant\'e, 
O-joung~Kwon, Sang-il Oum, Dani\"el Paulusma}
}
\usepackage{todonotes}
\usepackage{authblk}
\usepackage{amsthm,amssymb,amsmath}
\usepackage{blkarray}
\usepackage{mathabx}
\newtheorem{theorem}{Theorem}[section]
\newtheorem{lemma}[theorem]{Lemma}
\newtheorem{proposition}[theorem]{Proposition}
\newtheorem{corollary}[theorem]{Corollary}
\newtheorem{conjecture}[]{Conjecture}

\newcommand\abs[1]{\lvert #1\rvert}

\newcommand \dia{\hfill{$\diamond$}}

\newcommand\rank{\operatorname{rank}}
\newcommand\lrw{\operatorname{lrw}}
\newcommand\pw{\operatorname{pw}}

\newcommand\cutrk{\operatorname{cutrk}}

\def\bw_#1{{\overline{BW_{#1}}}}

\newcommand{\NP}{{\sf NP}}

\newcommand\pivot{\wedge}
\def\bw_#1{{\overline{BW_{#1}}}}

\begin{document}
\title{Tree Pivot-Minors and Linear Rank-Width\thanks{An extended abstract of this paper appeared in the proceedings of EuroComb 2019~\cite{DDJKKOP19}.
Dabrowski and Paulusma 
were supported by the Leverhulme Trust (RPG-2016-258). This work was mainly done when Jeong was in KAIST.  Kant\'e was supported by the French Agency for Research under the projects DEMOGRAPH (ANR-16-CE40-0028) and ASSK (ANR-18-CE40-0025). Kwon was supported by the National Research Foundation of Korea funded by the Ministry of Education (No. NRF-2018R1D1A1B07050294). Kwon and Oum were supported by the Institute for Basic Science (IBS-R029-C1).
Dross was supported by the European Research Council (ERC) under the European Union's Horizon 2020 research and innovation programme Grant Agreement 714704.}}

\author[1]{Konrad K. Dabrowski}
\affil[1]{School of Computing, University of Leeds, UK}

\author[2]{Fran\c{c}ois Dross}
\affil[2]{Institute of Informatics, University of Warsaw, Poland}

\author[3]{Jisu Jeong}
\affil[3]{Clova AI Research, NAVER Corp, Seongnam,~Korea}

\author[4]{Mamadou Moustapha Kant\'e} 
\affil[4]{Universit\'e Clermont Auvergne, LIMOS, CNRS, Aubi\`ere, France}

\author[5,6]{O-joung Kwon}
\affil[5]{Department of Mathematics, Incheon National University, Incheon,~Korea}

\author[6,7]{Sang-il Oum}
\affil[6]{Discrete Mathematics Group,
  Institute for Basic Science (IBS), Daejeon,~Korea}
\affil[7]{Department of Mathematical Sciences, KAIST, Daejeon,~Korea}
\author[1]{Dani\"el Paulusma}

\date\today
\maketitle

  \setcounter{footnote}{1}%
 \footnotetext{E-mail addresses: 
    \texttt{k.k.dabrowski@leeds.ac.uk} (Dabrowski),
    \texttt{francois.dross@googlemail.com} (Dross),
    \texttt{jisujeong89@gmail.com} (Jeong),
    \texttt{mamadou.kante@uca.fr} (Kant\'e),
    \texttt{ojoungkwon@gmail.com} (Kwon),
    \texttt{sangil@ibs.re.kr} (Oum),
    \texttt{daniel.paulusma@durham.ac.uk} (Paulusma).
    }

\vspace*{-1cm}
\begin{abstract}
Tree-width and its linear variant path-width play a central role for the graph minor relation.
In particular, Robertson and Seymour (1983)
proved that for every tree~$T$, the class of graphs that do not contain $T$ as a minor has bounded path-width.
For the pivot-minor relation, rank-width and linear rank-width take over the role of tree-width and path-width. As such, it is natural to examine if, for every tree~$T$, the class of graphs that do not contain $T$ as a pivot-minor has bounded linear rank-width.
We first prove that this statement is false whenever $T$ is a tree that is not a caterpillar.
We conjecture that the statement is true if $T$ is a caterpillar.
We are also able to give partial confirmation of this conjecture by proving:
\begin{itemize}
\item for every tree $T$, the class of $T$-pivot-minor-free distance-hereditary graphs has bounded linear rank-width if and only if $T$ is a caterpillar;
\item for every caterpillar $T$ on at most four vertices, the class of $T$-pivot-minor-free graphs has bounded linear rank-width.
\end{itemize}
\noindent
To prove our second result, we only need to consider $T=P_4$ and $T=K_{1,3}$, but we follow a general strategy: 
first we show that the class of $T$-pivot-minor-free graphs is contained in some class of $(H_1,H_2)$-free graphs, which we then show to have bounded linear rank-width. In particular, we prove that the class of 
$(K_3,S_{1,2,2})$-free graphs has bounded linear rank-width, which strengthens a known result that this graph class has bounded rank-width. 
\end{abstract}

\section{Introduction}\label{s-intro}

In order to increase our understanding of graph classes, it is natural to consider some notion of ``width'' and to research what properties graph classes of bounded width may have. We say that a graph class has {\it bounded} width (for some specific width parameter) if there exists a constant~$c$ such that the width of every graph in the class is at most~$c$.
In particular, this type of structural research has been done in the context of graph containment problems, where the aim is to determine whether a graph~$H$ appears as a ``pattern'' inside some other graph~$G$. 
Here, a pattern is defined by specifying a set of graph operations that may be used to obtain~$H$ from $G$.
For instance, a graph~$G$ contains a graph~$H$ as a {\em minor} if~$H$ can be obtained from~$G$ via a sequence of vertex deletions, edge deletions and edge contractions.

Tree-width and its {\it linear} variant, path-width, are the best-known graph width parameters due to their relevance for graph minor theory~\cite{GMXX}.
Rank-width is another well-known parameter, introduced by Oum and Seymour~\cite{OS06}. 
The rank-width of a graph~$G$ expresses the minimum width~$k$ of a tree-like structure obtained by recursively splitting the vertex set of $G$ in such a way that each cut induces a matrix of rank at most~$k$ (see Section~\ref{s-pre} for a formal definition).
Rank-width is more general than tree-width in the sense that every graph class of bounded tree-width has bounded rank-width, but there are classes for which the reverse does not hold, for example, the class of all complete graphs~\cite{CO00}.

The notion of rank-width has important algorithmic implications, as many \NP-complete decision problems are known to be polynomial-time solvable not only for graph classes of bounded tree-width, but also for graph classes of bounded rank-width; see~\cite{CourcelleMR00,EGW01,GK03,KR03,Ra07} for a number of meta-theorems capturing such decision problems.
Rank-width is equivalent to clique-width~\cite{OS06}, another important and well-studied width parameter.
Linear rank-width is a linearized variant of rank-width, known to be equivalent to linear clique-width (see, for example,~\cite{Ou17}) and to be closely related to the trellis-width of linear codes~\cite{Ka08}. We formally define the notions of rank-width and linear rank-width in Section~\ref{s-pre}.

The problem of determining whether a given graph has linear rank-width at most~$k$ for some given integer~$k$ is \NP-complete (this follows from a result of Kashyap~\cite{Ka08}).
On the positive side, Jeong, Kim, and Oum~\cite{JKO16} gave an FPT algorithm for deciding whether a graph has linear rank-width at most~$k$. 
Ganian~\cite{Ga11} and Adler, Farley, and Proskurowski~\cite{AFP13} characterized the graphs of linear rank-width at most~$1$.
Recently, Ne{\v{s}}et{\v{r}}il et al.~\cite{Nesetril19} showed that every class of bounded linear rank-width is linearly $\chi$-bounded.
However, our knowledge on linear rank-width, the topic of this paper, is still limited.

\subsection*{Motivation}

To increase our understanding of rank-width and linear rank-width, we may want to verify if classical results for tree-width and path-width stay valid when we replace tree-width with rank-width and path-width with linear rank-width. The following two structural results, related to path-width and tree-width, form the core of the Graph Minor Structure Theorem. Here, a graph~$G$ is {\it $H$-minor-free} for some graph~$H$ if $G$ does not contain $H$ as a minor.

\begin{theorem}[Robertson and Seymour~\cite{GMI}]\label{t-rs83}
For every tree~$T$, the class of $T$-minor-free graphs has bounded path-width.
\end{theorem}

\begin{theorem}[Roberson and Seymour~\cite{GMV}]\label{t-rs94}
For every planar graph~$H$, the class of $H$-minor-free graphs has bounded tree-width.
\end{theorem}

\noindent
It is known that edge deletions and contractions may increase the rank-width and linear rank-width~\cite{Co14}.
Hence, working with minors is not a suitable approach for understanding rank-width and linear rank-width. Therefore, Oum~\cite{Ou05} proposed the notions of vertex-minors and pivot-minors, two closely related notions, which were called $\ell$-reductions and $p$-reductions, respectively, in~\cite{Bouchet1994}. Taking vertex-minors or pivot-minors does {\it not} increase the rank-width or linear rank-width of a graph~\cite{Ou05}. 

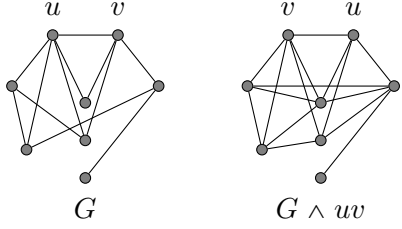
\begin{figure}
  \begin{center}
    \tikzstyle{v}=[circle,draw,fill=black!50,inner sep=0pt,minimum width=4pt]
  \begin{tikzpicture}
    \draw (0,0)node[v](v){};
    \foreach \i in {1,2,3,4,6,7} {
      \draw (90-360/14+360*\i/7:1) node [v] (v\i){};
      }
    \draw (0,-0.5) node [v] (v5){};
      \node at (v7) [label=above:$v$]{};
      \node at (v1) [label=above:$u$]{};
      \draw (v3)--(v2)--(v1)--(v7)--(v6)--(v4);
      \foreach \i in {7,1,2} {
      \draw (v5)--(v\i);
    }
    \draw (v7)--(v)--(v1);
    \draw (v1)--(v3)--(v6);
    \node at (v4) [label=below:$G$]{};
  \end{tikzpicture}
  $\quad\quad$
  \begin{tikzpicture}
    \draw (0,0)node[v](v){};
    \foreach \i in {1,2,3,4,6,7} {
      \draw (90-360/14+360*\i/7:1) node [v] (v\i){};
      }
    \draw (0,-0.5) node [v] (v5){};
      \node at (v7) [label=above:$u$]{};
      \node at (v1) [label=above:$v$]{};
      \draw (v3)--(v2)--(v1)--(v7)--(v6)--(v4);
      \foreach \i in {7,6,1,3} {
      \draw (v5)--(v\i);
    }
    \draw (v1)--(v3);
    \draw (v6)--(v2);
    \draw (v7)--(v)--(v1);
    \foreach \i in {2,3,6}{
      \draw (v\i)--(v);
    }
    \node at (v4) [label=below:$G\pivot uv$]{};
  \end{tikzpicture}
  \end{center}
\caption{A graph before and after pivoting an edge (the example is taken from~\cite{DDJKKOP18}).}
\label{f-pivotexample}
\end{figure}
	
To define the notions of a vertex-minor and a pivot-minor, we need some terminology. The \emph{local complementation} at a vertex~$u$ in a graph~$G$ replaces every edge of the subgraph induced by the neighbours of~$u$ with a non-edge, and vice versa.
The resulting graph is denoted by $G*u$.
An \emph{edge pivot} is the operation that takes an edge~$uv$, first applies a local complementation at~$u$, then at~$v$, and then at~$u$ again.
We denote the resulting graph $G\wedge uv=G*u*v*u$.
It is known that $G*u*v*u=G*v*u*v$~\cite{Ou05}, and thus $G\pivot uv=G \pivot vu$.
An alternative definition of the edge pivot operation is as follows.
Let~$S_u$ be the set of neighbours of~$u$ that are non-adjacent to~$v$ and let $S_v$ be the set of neighbours of~$v$ that are non-adjacent to~$u$, whereas we denote the set of common neighbours of~$u$ and~$v$ by $S_{uv}$.
We replace every edge between any two vertices in distinct sets from $\{S_u\setminus \{v\},S_v\setminus \{u\},S_{uv}\}$ by a non-edge and vice versa.
Afterwards, we delete every edge between~$u$ and~$S_u$ and add every edge between~$u$ and~$S_v$.
We also delete every edge between~$v$ and~$S_v$ and add every edge between~$v$ and~$S_u$.
We refer to \figurename~\ref{f-pivotexample} for an example.

A graph $H$ is a \emph{vertex-minor} of a graph $G$ if~$H$ can be obtained from $G$ by a sequence of local complementations and vertex deletions.
A graph $H$ is a \emph{pivot-minor} of a graph $G$ if~$H$ can be obtained from $G$ by a sequence of edge pivots and vertex deletions.
Hence $H$ is a vertex-minor of $G$ if $H$ is a pivot-minor of $G$, but the reverse is not necessarily true.
A graph is {\it $H$-vertex-minor-free} if it contains no vertex-minor isomorphic to $H$, and similarly, 
a graph is {\it $H$-pivot-minor-free} if it contains no pivot-minor isomorphic to $H$.

It is natural to ask whether parallel statements of Theorems~\ref{t-rs83} and \ref{t-rs94} exist for rank-width or linear rank-width in terms of vertex-minors or pivot-minors. 
Below we discuss the state-of-the-art for this research direction.

\subsection*{Related Work}

A \emph{circle graph} is the intersection graph of chords on a circle, and it is known that the class of circle graphs is closed under taking vertex-minors.
Bouchet~\cite{Bouchet1994} characterized circle graphs in terms of three forbidden vertex-minors. 
Oum~\cite{Ou05} showed that the class of circle graphs has unbounded rank-width, and 
asked, as an analogue to Theorem~\ref{t-rs94} for the vertex-minor relation,
whether for every circle graph $H$, the class of $H$-vertex-minor-free graphs has bounded rank-width. Recently, Geelen et al.~\cite{GKMW2019} gave an affirmative answer to this question. 

\begin{theorem}[Geelen, Kwon, McCarty, and Wollan~\cite{GKMW2019}]\label{thm:GKMW2019}
For every circle graph $H$, the class of $H$-vertex-minor-free graphs has bounded rank-width.
\end{theorem}

\noindent
Every pivot-minor of a graph is also a vertex-minor. Hence, for every graph~$H$, the class of $H$-vertex-minor-free graphs is contained in the class of $H$-pivot-minor-free graphs. This leads to the question whether we can strengthen Theorem~\ref{thm:GKMW2019} by replacing the vertex-minor relation with the pivot-minor relation. However, this is not the case.  
In order to see this, we first observe that
bipartite graphs are closed under taking pivot-minors~\cite{Ou05}. 
Hence, no bipartite graph contains a non-bipartite circle graph as a pivot-minor. 
Now consider the class~${\cal G}$ of $n\times n$ grids, which has unbounded rank-width.
As ${\cal G}$ is a subclass of bipartite graphs,
${\cal G}$ is $H$-pivot-minor-free for every non-bipartite circle graph~$H$ (such as, for example, $H=K_3$). Hence, for every non-bipartite graph $H$, the class of $H$-pivot-minor-free graphs has unbounded rank-width. This means we can only hope to strengthen Theorem~\ref{thm:GKMW2019} by considering bipartite circle graphs~$H$, and Oum~\cite{Oum2009} conjectured the following analogue to Theorem~\ref{t-rs94} for the pivot-minor relation:

\begin{conjecture}[Oum~\cite{Oum2009}]\label{conj1}
For every bipartite circle graph $H$, the class of $H$-pivot-minor-free graphs has bounded rank-width.
\end{conjecture} 

\noindent
So far, Conjecture~\ref{conj1} has been verified for bipartite graphs~\cite{Ou05}, circle graphs~\cite{Oum2009}, and line graphs~\cite{Oum2009}.
If Conjecture~\ref{conj1} holds for all graphs, then this would imply both Theorems~\ref{t-rs94} 
and~\ref{thm:GKMW2019}~\cite{Ou05}.

We now turn to linear rank-width, for which Kant\'e and Kwon~\cite{KanteK2018} conjectured the following analogue to Theorem~\ref{t-rs83} for the vertex-minor relation:

\begin{conjecture}[Kant\'e and Kwon~\cite{KanteK2018}]\label{conj2}
For every tree $T$, the class of $T$-vertex-minor-free graphs has bounded linear rank-width.
\end{conjecture}

\noindent
So far, Conjecture~\ref{conj2} has been verified on every class of graphs whose prime graphs, with respect to split decompositions, have bounded linear rank-width~\cite{KanteK2018}. For example, prime distance-hereditary graphs have at most three vertices, and therefore, for every tree $T$, the class of $T$-vertex-minor-free distance-hereditary graphs has bounded linear rank-width.
Moreover, Conjecture~\ref{conj2} holds for every path $T$~\cite{KMOW19}.

\subsection*{Our Focus and Results}

We focus on the remaining analogue, namely the analogue to Theorem~\ref{t-rs83} for the pivot-minor relation. We first prove that we cannot hope for a result that holds for every tree~$T$.
 A {\em caterpillar} is a tree that contains a path~$P$, such that every vertex not on~$P$ has a neighbour in~$P$.

\begin{theorem}\label{t-main1}
If~$T$ is a tree that is not a caterpillar, then the class of $T$-pivot-minor-free distance-hereditary graphs has unbounded linear rank-width.
\end{theorem}

\noindent
Due to Theorem~\ref{t-main1}, we conjecture the following:

\begin{conjecture}\label{q-1}
For every \emph{caterpillar} $T$, the class of $T$-pivot-minor-free graphs has bounded linear rank-width.
\end{conjecture}

\noindent
In contrast, the aforementioned result of Kwon et al.~\cite{KMOW19} confirming Conjecture~\ref{conj2} if $T$ is a path implies that Conjecture~\ref{conj2} {\it has} been confirmed if $T$ is a caterpillar: every caterpillar~$T$ is a pivot-minor of some path $P$~\cite[Theorem 4.6]{ko14} and consequently, if $T$ is a caterpillar, then the class of $T$-vertex-minor-free graphs is contained in the class of $P$-vertex-minor-free graphs.

By the fact that every caterpillar~$T$ is a pivot-minor of some path $P$ and the fact that every path $P$ is a caterpillar by definition, we can also formulate Conjecture~\ref{q-1} as follows:

\medskip
\noindent
{\bf Conjecture~\ref{q-1} (alternative formulation).}
{\it For every \emph{path} $P$, the class of $P$-pivot-minor-free graphs has bounded linear rank-width.}

\medskip
\noindent
We make two contributions to Conjecture~\ref{q-1}. We first show, in Section~\ref{pndh}, that Conjecture~\ref{q-1} holds for distance-hereditary graphs.

\begin{theorem}\label{thm:pmdh}
Let $n\ge 3$ be an integer. Every $P_n$-pivot-minor-free distance-hereditary graph has linear rank-width at most $2n-5$.
\end{theorem}

\noindent
Theorems~\ref{t-main1} and~\ref{thm:pmdh}, together with the fact that every caterpillar  is a pivot-minor of some path, yields the following dichotomy.

\begin{corollary}\label{c-dicho}
For every tree $T$, the class of $T$-pivot-minor-free distance-hereditary graphs has bounded linear rank-width if and only if $T$ is a caterpillar.
\end{corollary}

\noindent
If a graph $G$ is $P_4$-pivot-minor-free, then $G$ has no induced subgraph isomorphic to $P_4$. This 
implies 
that $G$ is distance-hereditary.
Hence, Theorem~\ref{thm:pmdh} has the following consequence:

\begin{corollary}\label{cor:p4pm}
Every $P_4$-pivot-minor-free graph has linear rank-width
at most~$3$.
\end{corollary}
\noindent
Below we give a short alternative proof of Corollary~\ref{cor:p4pm}
(without an explicit bound)
after introducing a more general strategy.
A graph~$G$ is {\it $H$-free} if $G$ does not contain the graph $H$ as an induced subgraph, and  $G$ is {\it $(H_1,\ldots,H_p)$-free} for some set of graphs $\{H_1,\ldots,H_p\}$ if $G$ is $H_i$-free for every $i\in \{1,\ldots,p\}$.
We can now try to obtain for a caterpillar~$T$, a constant bound on the linear rank-width of a $T$-pivot-minor-free graph by adapting the following {\it general} strategy:

\bigskip
\noindent
\begin{boxedminipage}{.92\textwidth}
{\bf Step 1.} Show that the class of $T$-pivot-minor-free graphs is a subclass of a class of\\
\hspace*{1.3cm} $(H_1,H_2)$-free graphs for some graphs $H_1$ and $H_2$.\\[3pt]
{\bf Step 2.} Show that this class of $(H_1,H_2)$-free graphs has bounded linear rank-width.
\end{boxedminipage}

\bigskip
\noindent
An advantage of this strategy is that it will lead to a stronger result that forms the start of a {\it systematic} study into boundedness of linear rank-width of $(H_1,H_2)$-free graphs.
This would address Open Problem~7.5 in~\cite{DJP19}, which asks for such a result. 
We refer to Section~\ref{sec:conclusion} for a further discussion on this. 

To illustrate our general strategy for the case where $H=P_4$, we can do as follows. 
In Step~1, we observe that every $P_4$-pivot-minor graph is $(P_4, \text{\textnormal{dart}})$-free
(see \figurename~\ref{f-three} for an illustration of the dart).
In Step~2, we use a result of Brignall, Korpelainen, and Vatter~\cite{BKV17}, who showed that
a class of $(P_4,H)$-free graphs has bounded linear rank-width if and only if $H$ is a threshold graph.
Hence, as the dart is a threshold graph,
the class of $(P_4, \text{\textnormal{dart}})$-free graphs, and thus the class of 
$P_4$-pivot-minor-free graphs, has bounded linear rank-width.

Whether $P_n$-pivot-minor-free graphs have bounded linear rank-width for $n\geq 5$ remains a challenging open question. In the remainder, we focus on the other tree $T$ on four vertices besides the $P_4$, which is the {\em claw} $K_{1,3}$ (the $4$-vertex star). 
We will prove the following result.

\begin{theorem}\label{t-claw}
Every claw-pivot-minor-free graph has linear rank-width at most~$59$.
\end{theorem}

\noindent
As a consequence, we have verified (the original formulation of) Conjecture~\ref{q-1} for every caterpillar~$T$ on at most four vertices. Since every tree on at most four vertices is a caterpillar we have in fact shown that if $T$ is a tree on at most four vertices, then the class of $T$-pivot-minor-free graphs has bounded linear rank-width.

We first explain how we perform Step~1. 
In our previous paper~\cite{DDJKKOP18}, 
we proved that a graph is claw-pivot-minor-free if and only if it is $(\mbox{bull},\mbox{claw}, P_5, W_4, \bw_3)$-free; see  \figurename~\ref{f-three} for pictures of these forbidden induced graphs.
It is readily seen that to prove boundedness of linear rank-width of some graph class~${\cal G}$ one may restrict to connected graphs in ${\cal G}$.
In~\cite{DDJKKOP18} we showed that a graph~$G$ is $(\text{\textnormal{bull}},\text{\textnormal{claw}},P_5)$-free if and only if every component of~$G$ is $3P_1$-free (the graph $3P_1$ consists of three isolated vertices).
Hence, we derive the following result, in which we specify the graphs $H_1$ and $H_2$ of Step~1 as
$H_1=3P_1$ and $H_2=W_4$.

\begin{figure}
\begin{center}
\tikzstyle{v}=[circle,draw,fill=black!50,inner sep=0pt,minimum width=4pt]
\begin{tikzpicture}
\draw (0.5,1) node [v] (v){};
\draw (0,0) node [v] (v1){};
\draw (0.5,0) node [v] (v2){};
\draw (1,0) node [v] (v3){};
\draw (v1)--(v)--(v2) (v)--(v3);
\draw (0.5,0) node[label=below:claw]{};
\end{tikzpicture}
$\quad\quad$
\begin{tikzpicture}
\draw (0,0) node [v] (v1){};
\draw (1,0) node [v] (v2){};
\draw (1,1) node [v] (v3){};
\draw (0,1) node [v] (v4){};
\draw (.5,.5) node[v] (v) {};
\draw (v4)--(v1)--(v)--(v2)--(v3);
\draw (0.5,0) node[label=below:$P_5$]{};
\end{tikzpicture}
$\quad\quad$
\begin{tikzpicture}
\draw (0,0) node [v] (v1){};
\draw (1,0) node [v] (v2){};
\draw (1,1) node [v] (v3){};
\draw (0,1) node [v] (v4){};
\draw (.5,.5) node[v] (v) {};
\draw (v4)--(v1)--(v)--(v2)--(v3);
\draw (v1)--(v2);
\draw (0.5,0) node[label=below:bull]{};
\end{tikzpicture}
$\quad\quad$
\begin{tikzpicture}
\draw (0,0) node [v] (v1){};
\draw (1,0) node [v] (v2){};
\draw (1,1) node [v] (v3){};
\draw (0,1) node [v] (v4){};
\draw (.5,.5) node[v] (v) {};
\foreach \i in {1,2,3,4} {
\draw (v)--(v\i);
}
\draw (v1)--(v2)--(v3)--(v4)--(v1);
\draw (0.5,0) node[label=below:$W_4$]{};
\end{tikzpicture}
$\quad\quad$
\begin{tikzpicture}
\draw (0,0.25) node [v] (v1){};
\draw (1,0.25) node [v] (v2){};
\draw (0.5,0.5) node [v] (v3){};
\draw (0.5,0) node [v] (v4){};
\draw (0.5,1) node[v] (v) {};
\draw (v1)--(v4)--(v2)--(v3)--(v1);
\draw (v)--(v3)--(v4);
\draw (0.5,0) node[label=below:dart]{};
\end{tikzpicture}
$\quad\quad$
\begin{tikzpicture}
\begin{scope}[xshift=-1.5cm]
\foreach \i in {1,2} {
\draw (120*\i:.7) node[v](a\i){};
}
\foreach \i in {3} {
\draw (-10:.7) node[v](a\i){};
}
\draw (a1)--(a2)--(a3)--(a1);
\end{scope}
\begin{scope}[xshift=1.5cm]
\draw (10:1) node[v] (b4){};
\foreach \i in {1,2} {
\draw (180-120*\i:.7) node[v](b\i){};
\draw (a\i)--(b\i)--(b4);
}
\foreach \i in {3} {
\draw (190:.7) node[v](b\i){};
\draw (a\i)--(b\i)--(b4);
}
\draw (b1)--(b2)--(b3)--(b1);
\end{scope}
\draw (-90:.4) node[label=below:$\bw_3$]{};
\end{tikzpicture}
\end{center}
\caption{The graphs claw, $P_5$, bull, $W_4$, dart, and $\bw_3$.}\label{f-three}
\end{figure}
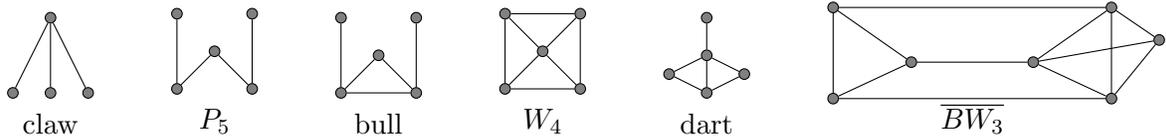

\begin{theorem}[Dabrowski et al.~\cite{DDJKKOP18}]\label{t-known}
Let $G$ be a connected graph. Then $G$ is claw-pivot-minor-free if and only if $G$ is $(3P_1,W_4,\bw_3)$-free. In particular,
the class of connected claw-pivot-minor-free graphs belongs to the class of $(3P_1,W_4)$-free graphs.
\end{theorem}

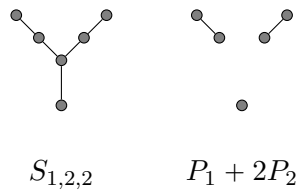
\begin{figure}[b]
   \tikzstyle{v}=[circle,draw,fill=black!50,inner sep=0pt,minimum width=4pt]
  \centering
   \begin{tikzpicture}[scale=0.6]
        \node [v]  (b1) at (1-4, 0-10){};
        \node [v]  (b3) at (2-4, 2-10){};
        \node [v]  (b4) at (0-4, 2-10){};
        \node [v]  (b5) at (1-4, 1-10){};
        \node [v]  (b6) at (0.5-4, 1.5-10){};
        \node [v]  (b7) at (1.5-4, 1.5-10){};
        
        \draw(b5)--(b6)--(b4);
        \draw(b5)--(b7)--(b3);
        \draw(b5)--(b1);
         \draw (1-4,-1-10) node [below] {$S_{1,2,2}$}; 
  
        \node [v]  (d1) at (1, 0-10){};
        \node [v]  (d3) at (2, 2-10){};
        \node [v]  (d4) at (0, 2-10){};
        \node [v]  (d6) at (0.5, 1.5-10){};
        \node [v]  (d7) at (1.5, 1.5-10){};
        
        \draw(d6)--(d4);
        \draw(d7)--(d3);
           \draw (1,-1-10) node [below] {$P_1+2P_2$}; 
\end{tikzpicture}
   \caption{The graphs $S_{1,2,2}$ and $P_1+2P_2$; note that $P_1+2P_2$ is an induced subgraph of $S_{1,2,2}$.}
  \label{fig:specialgraph}
\end{figure}

\noindent
As Step~2, we must prove that $(3P_1,W_4)$-free graphs have bounded linear rank-width.
With an eye on a future classification of boundedness of linear rank-width for $(H_1,H_2)$-free graphs, we aim to prove boundedness for classes of $(H_1,H_2)$-free graphs that are as large as possible. 
For integers $1\leq i\leq j\leq k$, let~$S_{i,j,k}$ denote the {\it subdivided claw}, which is the graph obtained from the claw by subdividing its three edges
$i-1$ times, $j-1$ times and $k-1$ times, respectively; see also \figurename~\ref{fig:specialgraph} and note that $S_{1,1,1}=K_{1,3}$. 
The {\it complement} of a graph~$G$ is the graph $\overline{G}$ with $V(\overline{G})=V(G)$ and $E(\overline{G})=\{uv\; |\; u,v\in V(G)\; \mbox{with}\; u\neq v \; \mbox{and}\; uv\notin E(G)\}$.
The graph $K_3$ denotes the triangle.
As $3P_1=\overline{K_3}$ and $W_4=\overline{P_1+2P_2}$ is an induced subgraph of 
the complement of $\overline{S_{1,2,2}}$, the class of $(3P_1,W_4)$-free graphs is contained in the class of $(\overline{K_3},\overline{S_{1,2,2}})$-free graphs.
We prove the following result in Section~\ref{claw}:

\begin{theorem}\label{t-maintwosubgraphs}
Every $(K_3, S_{1,2,2})$-free graph has linear rank-width at most~$58$.
\end{theorem}

\noindent
As observed in Lemma~\ref{lem:complement} in Section~\ref{s-pre}, complementing a graph may increase its linear rank-width by at most~$1$. Hence, Theorems~\ref{t-known} and~\ref{t-maintwosubgraphs} and this observation imply Theorem~\ref{t-claw}.

Dabrowski et al.~\cite{DabrowskiDP2016} proved that the class of $(K_3, S_{1,2,2})$-free graphs has bounded rank-width. As every class of bounded linear rank-width has bounded rank-width, but the reverse is not necessarily true, Theorem~\ref{t-maintwosubgraphs} is a strengthening of their result.
Moreover, Theorem~\ref{t-maintwosubgraphs} is tight in the sense that even
the class of $S_{1,2,3}$-free bipartite graphs is known to have unbounded linear rank-width~\cite{ALZ18}.

It remains to prove Theorems~\ref{t-main1},~\ref{thm:pmdh} and~\ref{t-maintwosubgraphs}, which we do in Sections~\ref{sec:caterpillar},~\ref{pndh} and~\ref{claw}, respectively.
We discuss future research in Section~\ref{sec:conclusion}.

\section{Preliminaries}\label{s-pre}

	In this paper, all graphs have no loops and no multiple edges.
	For a graph~$G$, let~$V(G)$ and~$E(G)$ denote the vertex set and edge set of~$G$, respectively. 
	For $S\subseteq V(G)$, let~$G[S]=(S,\{uv\; |\; uv\in E, u,v\in S\})$ denote the subgraph of~$G$ induced by~$S$. 
	For convenience, we write $G[v_1, v_2, \ldots, v_m]$ for $G[\{v_1, v_2, \ldots, v_m\}]$.
	A graph~$H$ is an \emph{induced subgraph} of $G$ if $H=G[S]$ for some $S\subseteq V(G)$.
	For a vertex $v\in V(G)$, we let $G- v$ be the graph obtained from~$G$ by removing~$v$. 
	For a set $S\subseteq V(G)$, we let $G-S$ be the graph obtained from~$G$ by removing all vertices in~$S$. 
	For an edge $e\in E(G)$, we let $G-e$ be the graph obtained from~$G$ by removing~$e$. For  
	a set $F\subseteq E(G)$, we let $G-F$ be the graph obtained from~$G$ by removing all edges in $F$.

The set of neighbours of a vertex~$v$ in a graph~$G$ is denoted by~$N_G(v)$. 
The size of~$N_G(v)$ is the \emph{degree} of~$v$.
	For a set $A\subseteq V(G)$, we let~$N_G(A)$ denote the set of all vertices in $V(G)\setminus A$ that have a neighbour in~$A$.
	Two vertices~$v$ and~$w$ in~$G$ are \emph{twins} if $N_G(v)\setminus \{w\}=N_G(w)\setminus \{v\}$.
	We say that twins~$v$ and~$w$ are \emph{false twins} if~$v$ is not adjacent to~$w$.
An edge~$e$ of a connected graph~$G$ is a \emph{cut edge} if $G-e$ is disconnected.
	
Let $A$ and $B$ be two disjoint vertex subsets of a graph $G$.
We let $G\times (A,B)$ be the graph  obtained from $G$ by taking a \emph{bipartite complementation} between $A$ and $B$, that is, by replacing each edge between a vertex of $A$ and a vertex of $B$ by a non-edge, and vice versa.
We say that~$A$ is \emph{complete} to~$B$ if $a$ is adjacent to~$b$ for every $a\in A$ and every $b\in B$, whereas~$A$ is \emph{anti-complete} to~$B$ if $a$ is not adjacent to~$b$ for every $a\in A$ and every $b\in B$.
If~$A$ is complete or anti-complete to~$B$, then $A$ is \emph{trivial} to~$B$.
If $A$ consists of one vertex~$v$, then we may say that~$v$ is complete or anti-complete to~$B$.

A set~$F$ of edges is a \emph{matching} in a graph $G$ if no two edges in~$F$ have a common end-vertex.  A set $S$ of vertices in a graph $G$ is an \emph{independent set} if no two vertices in $S$ are adjacent, 
whereas $S$ is a \emph{clique} if every pair of vertices in $S$ is adjacent.	
The complete graph $K_n$ is the graph on $n$ vertices that form a clique.  
The complete bipartite graph~$K_{n,m}$ is the bipartite graph
with a bipartition $(A,B)$ such that $\abs{A}=n$, $\abs{B}=m$, and~$A$ is complete to~$B$.  The graph~$W_n$ is the graph on $n+1$ vertices that is obtained from a cycle on $n$ vertices by
adding one vertex that is made adjacent to all vertices in the cycle.
The \emph{length} of a path is the number of edges in the path.

For two graphs~$G$ and~$H$ on disjoint vertex sets, we let $G+\nobreak H$ be the disjoint union of~$G$ and~$H$, which has $V(G+H)=V(G)\cup V(H)$ and $E(G+H)=E(G) \cup E(H)$.
We let $pG$ denote the disjoint union of~$p$ copies of~$G$. 
The {\it subdivision} of an edge $uv$ in a graph removes the edge $uv$ and introduces a new vertex $w$ that is made adjacent (only) to $u$ and $v$.

A graph~$H$ is the \emph{$1$-subdivision} of a graph~$G$ if~$H$ is obtained from~$G$ by subdividing each edge of $G$ exactly once.
 A graph~$G$ is \emph{distance-hereditary} if for every connected induced subgraph~$H$ of~$G$ and every two vertices $v,w$ in~$H$, the distance between~$v$ and~$w$ in~$H$ is the same as the distance in~$G$. 
 
Let~$G$ be a graph with vertices $x_1,\ldots,x_n$.
Let~$A=A_G$ denote the \emph{adjacency matrix} of~$G$, that is, entry $A_G(i,j)=1$ if $x_i$ is adjacent to $x_j$ and $A_G(i,j)=0$ if $x_i$ is not adjacent to $x_j$.
For a subset $X\subseteq V(G)$, the matrix $A[X,V(G)\setminus X]$ is the $|X|\times |V(G)\setminus X|$ submatrix of $A$ restricted to the rows of $X$ and the columns of $V(G)\setminus X$. 
The \emph{cut-rank function} of~$G$ is the function $\cutrk_G:2^{V(G)}\rightarrow \mathbb{N}$ such that for each $X\subseteq V(G)$,  \[\cutrk_G(X):=\rank(A_G[X,V(G)\setminus X]),\] where we compute the rank {\it over the binary field}.
A {\it linear ordering} of $G$ is a permutation of the vertices of $G$.
The \emph{width} of a linear ordering $(x_1,\ldots, x_n)$ of~$G$ is defined as $\max_{1\le i\le n}\{\cutrk_G(\{x_1,\ldots,x_i\})\}$.
The \emph{linear rank-width} $\lrw(G)$ of $G$ is the minimum width over all linear orderings of~$G$.

Let $X$ and $Y$ be two disjoint subsets of vertices of a graph.
For an ordering $(x_1, \ldots, x_n)$ of the vertices of $X$ and an ordering $(y_1, \ldots, y_m)$ of the vertices of $Y$, we define the ordering
	\[(x_1, \ldots, x_n)\oplus (y_1, \ldots, y_m):=(x_1, \ldots, x_n, y_1, \ldots, y_m).\]

The cut-rank function is invariant under taking local complementation. This implies that the linear rank-width of a graph does not increase when taking its vertex-minor.

\begin{lemma}[Bouchet~\cite{Bouchet1989}; See Oum~\cite{Ou05}]\label{lem:bouchet}
If $G$ is obtained from $H$ by a sequence of local complementations, then 
$\cutrk_G(X)=\cutrk_H(X)$ for all $X\subseteq V(G)$. 
So, if $G$ is a vertex-minor of $H$, then $\lrw(G)\le \lrw(H)$.
\end{lemma}

We need three structural lemmas on linear rank-width.
Recall that the complement of a graph~$G$ is the graph $\overline{G}$ with $V(\overline{G})=V(G)$ and $E(\overline{G})=\{uv\; |\; u,v\in V(G)\; \mbox{with}\; u\neq v \; \mbox{and}\; uv\notin E(G)\}$.

\begin{lemma}\label{lem:complement}
	If~$G$ has linear rank-width~$k$, then~$\overline{G}$ has linear rank-width at most~$k+\nobreak 1$.
	\end{lemma}
	
	\begin{proof}
	Let~$H$ be the graph obtained from~$G$ by adding a vertex $a$ complete to $V(G)$.
	Observe that $\overline{G}=(H*a)[V(G)]$. Then, $\lrw(\overline{G})\leq \lrw(H*a)=\lrw(H)\le \lrw(G)+1$.
	Here, the second step follows from Lemma~\ref{lem:bouchet}, and the third step follows from the fact that adding one vertex to a graph may increase the linear rank-width by at most 
	one.	\end{proof}
		
\begin{lemma}\label{lem:bipartitecomplement}
	Let~$G$ be a graph and~$A$ and~$B$ be two disjoint vertex subsets of~$G$.
	If~$G$ has linear rank-width~$k$, then $G\times (A,B)$ has linear rank-width at most~$k+\nobreak 2$.
	\end{lemma}
	\begin{proof}
	Let~$H$ be the graph obtained from~$G$ by adding two adjacent vertices~$a$ and~$b$ such that
	\begin{itemize}
	\item $N_H(a)\cap V(G)=A$ and $N_H(b)\cap V(G)=B$.
	\end{itemize}
	Observe that $G\times (A,B)=(H\pivot ab)[V(G)]$.
	Since adding two vertices may increase the linear rank-width by at most two, 
	we have $\lrw(H\pivot ab)=\lrw(H)\le \lrw(G)+2$.
	Therefore, we have $\lrw(G\times (A,B))\le \lrw(G)+2$.
	\end{proof}
	
Let~$I$ be a set of pairwise twins in a graph~$G$.
We define~$G//I$ as the graph obtained from~$G$ by removing all the vertices of~$I$ except one vertex.

\begin{lemma}\label{lem:twinreduced}
Let~$G$ be a graph and $I_1, I_2, \ldots, I_m$ be pairwise disjoint subsets of~$V(G)$ such that each~$I_i$ is a set of pairwise twins in~$G$.
Then $\lrw(G)\le \lrw(G//I_1 //I_2 // \cdots //I_m)+1$.
\end{lemma}

\begin{proof}
Let $H:=G//I_1 //I_2 //\cdots //I_m$.
For each $j\in \{1, \ldots, m\}$, let~$w_j$ be the vertex kept from~$I_j$ in~$H$.
Let $I:=\bigcup_{j\in \{1, \ldots, m\}}I_j$.
Suppose that~$L_H$ is a linear ordering of~$H$ with the optimal width. 
We obtain a linear ordering $L_G$ of~$G$ from~$L_H$ by replacing each $w_j\in \{w_1, \ldots, w_m\}$ with any linear ordering of~$I_j$.
Let $L_G:=(v_1, v_2, \ldots, v_n)$. 
We claim that $L_G$ has width at most $\lrw(H)+1$. Let $i\in \{1, \ldots, n\}$, and let $L:=\{v_j \mid 1\le j\le i\}$ and $R:=V(G)\setminus L$.
It suffices to show that $\rank(A(G)[L, R])\le \lrw(H)+1$.
We will use the fact that 
\begin{itemize}
\item[($\ast$)] if two rows of a matrix $M$ are the same, then the matrix obtained from $M$ by removing one of these rows has the same rank as $M$, and the same argument holds for columns.
\end{itemize}
We observe that at most one set of $I_1, I_2, \ldots, I_m$ may have a vertex in both $L$ and $R$.
Let $J_1:=\{i\in \{1, 2, \ldots, m\} \mid I_i\cap L\neq \emptyset\}$ and $J_2:=\{i\in \{1, 2, \ldots, m\} \mid I_i\cap R\neq \emptyset\}$, and 
let $W_1:=\{w_i \mid i\in J_1\}$ and $W_2:=\{w_i \mid i\in J_2\}$.
First assume that no set of $I_1, I_2, \ldots, I_m$ has a vertex in both $L$ and $R$.
By ($\ast$), we have 
$\rank\big(A(G)[L, R]\big)
=\rank\big(A(G)[(L\setminus I)\cup W_1, (R\setminus I)\cup W_2]\big)$.
As the partition $((L\setminus I)\cup W_1, (R\setminus I)\cup W_2)$ is considered when computing the width of $L_H$, we have $\rank(A(G)[L, R])\le \lrw(H)$.
Therefore, we may assume that there is a set $I_p$ of $I_1, \ldots, I_m$ having a vertex in both $L$ and $R$.
Let $x\in I_p\cap L$ and $y\in I_p\cap R$.
By ($\ast$), we have 
\begin{align*}
&\rank\big(A(G)[L, R]\big) \\
=&\rank\big(A(G)[(L\setminus I)\cup (W_1\setminus \{w_p\}) \cup \{x\}, (R\setminus I)\cup (W_2\setminus \{w_p\})\cup \{y\}]\big).
\end{align*}
Note that one of the partitions $((L\setminus I)\cup W_1, (R\setminus I)\cup (W_2\setminus \{w_p\}))$ 
and $((L\setminus I)\cup (W_1\setminus \{w_p\}), (R\setminus I)\cup W_2)$ is considered when computing the width of $L_H$. 
Since $x,y,w_p\in I_p$, the matrix 
\[A(G)[(L\setminus I)\cup (W_1\setminus \{w_p\}) \cup \{x\}, (R\setminus I)\cup (W_2\setminus \{w_p\})\cup \{y\}]\]
can be obtained from 
$A(G)[(L\setminus I)\cup W_1, (R\setminus I)\cup (W_2\setminus \{w_p\})]$ by adding one column corresponding to $y$, and 
it also can be obtained from 
$A(G)[(L\setminus I)\cup (W_1\setminus \{w_p\}), (R\setminus I)\cup W_2]$ by adding one row corresponding to $x$.
This implies that $\rank(A(G)[L, R])\le \lrw(H)+1$.
We conclude that $\lrw(G)\le \lrw(H)+1$.
\end{proof}

A \emph{path decomposition} of a graph $G$ is an ordered family ${\cal B}=\{B_1,\ldots,B_r\}$ of subsets of $V(G)$ 
satisfying the following:
\begin{enumerate}
\item For every $v\in V(G)$ there exists a $t\in \{1,\ldots,r\}$ with $v\in B_t$.
\item For every $uv\in E(G)$ there exists a $t\in \{1,\ldots,r\}$ with $\{u,v\}\subseteq B_t$.
\item For every $v\in V(G)$, the set $\{t\in \{1,\ldots,r\}\mid v\in B_t\}$ consists of consecutive integers.
\end{enumerate}
The \emph{width} of a path decomposition $(P,\mathcal{B})$ is defined as 
$\max\{|B_t|\mid t \in \{1,\ldots,r\}\}-1$.
The \emph{path-width} of $G$ is the minimum width among all path decompositions of $G$.

We finish this section by proving that every tree with linear rank-width~$1$ is a caterpillar 
(recall  that a caterpillar is a tree that contains a path~$P$, such that every vertex not on~$P$ has a neighbour in~$P$).

\begin{theorem}[Adler and Kant\'e~\cite{AdlerK2015}]\label{thm:lrwpwtree}
	For every tree $T$, the linear rank-width of a tree~$T$ is equal to the path-width of $T$. 
\end{theorem}

\begin{theorem}[Takahashi, Ueno, and Kajitani \cite{Takahashi1994}]\label{thm:charpw}
 Let $G$ be a tree
 and $k$ be a positive integer.  
 Then $G$ has path-width at most $k$ if and only if for every vertex $v$, $G-v$ has at most two connected components with path-width exactly $k$ and all other connected components of
$G-v$
have path-width less than $k$.
\end{theorem}

	\begin{lemma}\label{lem:charlrw1tree}
	A tree has linear rank-width at most~$1$ if and only if it is a caterpillar.
      \end{lemma}
 
      \begin{proof}
      Let $T$ be a tree. 
      First suppose that $T$ is not a caterpillar. Then $T$ contains $S_{2,2,2}$ as an induced subgraph. 
      Thus, $T$ contains a vertex $v$ such that $T-v$ contains at least three connected components each containing an edge.
      So, $T-v$ contains three connected components having path-width at least $1$.
      By Theorem~\ref{thm:charpw}, $T$ has path-width at least $2$, and by Theorem~\ref{thm:lrwpwtree}, 
      $T$ has linear rank-width at least $2$.
      
      Now suppose that $T$ is a caterpillar.
      We prove by induction on $\abs{V(T)}$ that $T$ has path-width at most~$1$. 
      We may assume that $T$ has at least two vertices.
      Note that for every vertex $v$, $T-v$ has at most two connected components having an edge, which are still caterpillars, 
      and all the other connected components are isolated vertices.
      So, $T-v$ has at most two connected components having path-width $1$ by induction, and all the other connected components have path-width $0$.
      Thus, by Theorem~\ref{thm:charpw}, $T$ has path-width at most $1$. This proves the claim.
       We now apply Theorem~\ref{thm:lrwpwtree} to conclude that every caterpillar has linear rank-width at most $1$.
      \end{proof}

\section{The Proof of Theorem~\ref{t-main1}}\label{sec:caterpillar}

In this section, we prove Theorem~\ref{t-main1}, which states that for every tree $T$ that is not a caterpillar, the class of $T$-pivot-minor-free graphs has unbounded linear rank-width.

		\begin{figure}[t]
	\centerline{
	\includegraphics[scale=0.6]{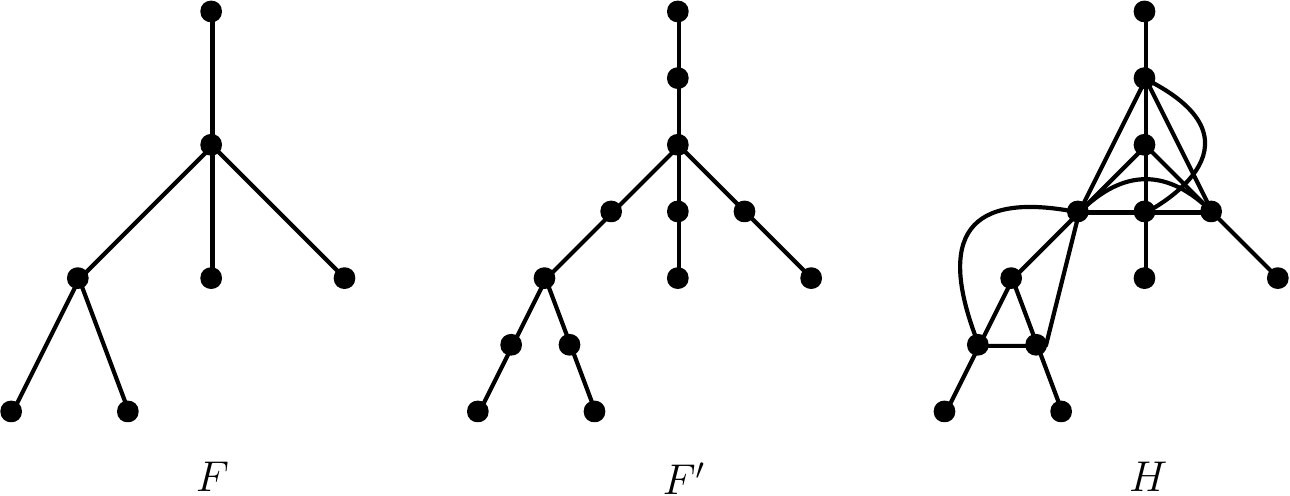}}
	\caption{A tree $F$, the tree $F'$ that is obtained from $F$ by subdividing each edge of $F$ once, and the graph 
	$H\in {\cal C}$ that is obtained from $F'$ after  applying a local complementation at every vertex of degree at least $3$ in $F'$.}
	\label{fig:classt}
	\end{figure}
	
	Let~$\mathcal{C}$ be the class of graphs that can be obtained from the~$1$-subdivision of a tree by applying a local complementation at every vertex of degree at least~$3$.  	
	We give an example of a graph in ${\cal C}$ in \figurename~\ref{fig:classt}.
Our proof of Theorem~\ref{t-main1} consists of the following parts:

\begin{enumerate}
\item we show that ${\cal C}$ has unbounded linear rank-width;
\item we show that every graph in ${\cal C}$ is a distance-hereditary graph with some additional properties needed to prove the third step; and
\item we show that  every graph in ${\cal C}$ is $T$-pivot-minor-free whenever $T$ is a tree that is not a caterpillar.
\end{enumerate}

We start with the following lemma that proves the first part.

\begin{lemma}\label{l-c}
The class~${\cal C}$ has unbounded linear rank-width.
\end{lemma}
\begin{proof}
Adler and Kant\'e~\cite{AdlerK2015} proved that trees have unbounded linear rank-width.
Note that the 1-subdivision $H$ of a graph $G$ contains $G$ as a vertex-minor:  for every subdivided vertex in $H$, perform a local complementation and remove it; this yields $G$.
So, by Lemma~\ref{lem:bouchet}, 
the class of $1$-subdivisions of trees also has unbounded linear rank-width. 
	As local complementations do not change the linear rank-width of a graph by Lemma~\ref{lem:bouchet},  this means that
	$\mathcal{C}$ has unbounded linear rank-width.
\end{proof}

We will now prove that ${\cal C}$ is a subclass of the class of distance-hereditary graphs with some additional useful properties.
	In order to do this, we need the notion of a canonical split decomposition of a graph~\cite{CunninghamE80}, which we define below.

\begin{figure}
\centerline{\includegraphics[scale=0.5]{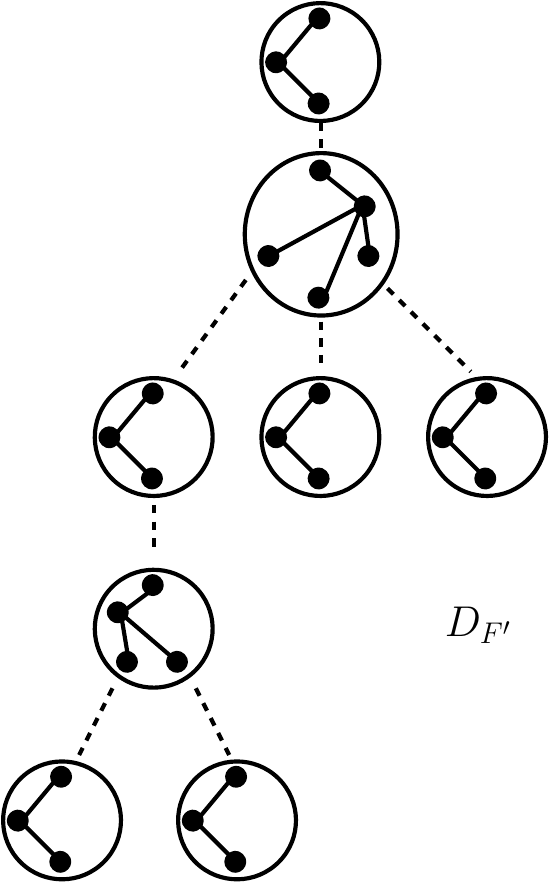} \quad\quad
\includegraphics[scale=0.5]{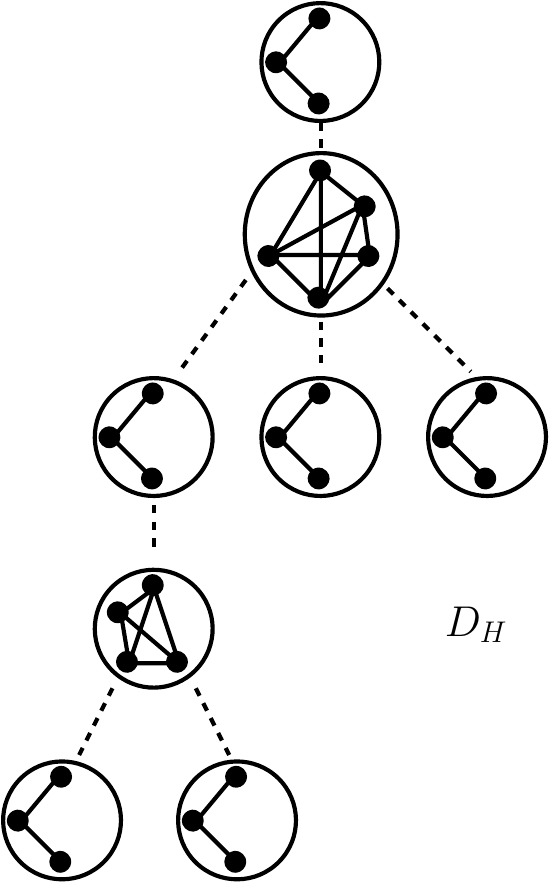} }
\caption{For the graphs $F'$ and $H$ in \figurename~\ref{fig:classt}, the decompositions $D_{F'}$ and $D_H$ are the canonical split decompositions of $F'$ and $H$, respectively.
Dashed edges denote marked edges and each circle denotes a bag.
Note that by the definition of a split, every bag contains at least three vertices.}
\label{fig:decomposition}
\end{figure}

	A vertex partition~$(X,Y)$ of a connected graph $G$ is a \emph{split} of~$G$ if $\abs{X}\ge 2, \abs{Y}\ge 2$, 
	and~$N_G(Y)$ is complete to~$N_G(X)$. 
	A connected graph $G$ on at least five vertices is \emph{prime} if it has no split.
		A connected graph~$D$ with a distinguished set of edges~$M(D)$ is a \emph{marked graph} if~$M(D)$ is a matching and each edge in~$M(D)$ is a cut edge.
	An edge in~$M(D)$ is a \emph{marked edge}, and every other edge of $D$ is an \emph{unmarked edge}.
	A vertex incident with a marked edge is a \emph{marked vertex}, and every other vertex of $D$ is an \emph{unmarked vertex}.	
	Each connected component of $D-M(D)$ is a \emph{bag} of~$D$. If a marked edge $e$ is incident with a vertex of a bag~$B$, we say that $B$ is
        \emph{incident} with $e$.  A bag $B_1$ of $D$ is a neighbour bag of a bag $B_2$ of $D$ if there is a marked edge incident with both $B_1$ and $B_2$.
        The \emph{decomposition tree} of $D$ is the graph obtained from $D$ by contracting each bag into a vertex.

	If~$G$ has a split~$(X,Y)$, we construct a marked graph~$D$ on the vertex set $V(G) \cup \{x_1,y_1\}$ for some new vertices~$x_1$ and~$y_1$ such that
	\begin{itemize}
	\item for every two
	distinct vertices $x,y$ with $\{x,y\}\subseteq X$ or $\{x,y\}\subseteq Y$, the property $xy\in E(G)$ holds if and only if $xy\in E(D)$,
	\item $x_1y_1$ is a new marked edge,
	\item $X$ is anti-complete to~$Y$,
	\item $x_1$ is complete to $N_G(Y)$ (with only unmarked edges) and has no neighbours in $V(G)\setminus N_G(Y)$,
	\item $y_1$ is complete to $N_G(X)$ (with only unmarked edges) and has no neighbours in $V(G)\setminus N_G(X)$. 
	\end{itemize}
	The graph~$D$ is a \emph{simple decomposition} of~$G$.
	To obtain a split decomposition, we will recursively take a simple decomposition of a bag, 
	and when we take a simple decomposition of a bag, all the marked vertices remain marked vertices.
That is, a	
	\emph{split decomposition} of a connected graph~$G$ is a marked graph~$D$ defined inductively to be either~$G$ or 
	a marked graph obtained from a split decomposition~$D'$ of~$G$ 
by replacing some bag $B$ of $D'$ by the bags of a simple decomposition $B'$ of $B$ and keeping all the marked edges between vertices of $V(B)$ and vertices of $V(D')\setminus V(B)$.
We give an example of a split decomposition in \figurename~\ref{fig:decomposition}.	

	For a marked edge $xy$ of a marked graph $D$, the \emph{recomposition
  of $D$ along $xy$} is the marked graph $(D\wedge xy) -\{x,y\}$, where when we pivot $xy$, we add unmarked edges between $N_D(x)\setminus \{y\}$ and $N_D(y)\setminus \{x\}$.  This operation can be seen as merging two adjacent bags $B_1$ and $B_2$ into one bag $B$
  where the union of $B_1$ and $B_2$ was a simple decomposition of $B$.
  It is not hard to see that  if $D$ is a split decomposition of $G$, then $G$ can be obtained from $D$ by 
recomposing along all the marked edges.  

	A split decomposition~$D$ is \emph{canonical} if each bag of~$D$ is either a prime graph, a star, or a complete graph, and 
for every marked edge~$xy$ in~$D$, recomposing~$xy$
results in a split decomposition having a bag that is neither a prime graph, a star, nor a complete graph. 	
We note that the split decompositions in \figurename~\ref{fig:decomposition} are canonical.
	We say that a bag is a {\it star bag} if it is a star and a {\it complete} bag if it is a complete graph.

	\begin{theorem}[Cunningham and Edmonds~\cite{CunninghamE80}] \label{thm:CED} 
	Every connected graph has a unique canonical split decomposition, up to isomorphism.
	\end{theorem}

Bouchet~\cite{Bouchet1988} described how split decompositions change under applying local complementations.
A vertex~$v$ in a split decomposition~$D$ \emph{represents} an unmarked vertex~$x$ (or is a \emph{representative} of~$x$) if either $v=x$ or there is a path of even length from~$v$ to~$x$ in~$D$ starting with a marked edge such that marked edges and unmarked edges appear alternately in the path. 
Observe that for every vertex~$v$ in a split decomposition $D$ of a graph $G$, there exists 
a vertex 
$x\in V(G)$ such that $v$ represents $x$.
Two unmarked vertices~$x$ and~$y$ are \emph{linked} in~$D$ if there is a path from~$x$ to~$y$ in~$D$ such that unmarked edges and marked edges appear alternately in the path.

A \emph{local complementation} at an unmarked vertex~$x$ in a split decomposition~$D$, denoted by~$D*x$, is the operation that replaces each bag~$B$ containing a representative~$w$ of~$x$ with~$B*w$. 
	
\begin{lemma}[Bouchet~\cite{Bouchet1988}]\label{lem:localdecom} 
Let~$D$ be the canonical split decomposition of a connected graph~$G$.
If~$x$ is an unmarked vertex of~$D$, then~$D*x$ is the canonical split decomposition of~$G*x$.
\end{lemma}

Let~$x$ and~$y$ be linked unmarked vertices in a split decomposition~$D$, and let~$P$ be the path in~$D$ linking~$x$ and~$y$ where unmarked edges and marked edges appear alternately.
Observe that such a path is unique.  
The \emph{pivoting on~$xy$ of~$D$}, denoted by $D\wedge xy$, is the split decomposition obtained as follows: for each bag~$B$ containing an unmarked edge~$vw$ of~$P$, we replace~$B$ with $B\wedge vw$. 

\begin{lemma}[Adler, Kant\'e, and Kwon~\cite{AdlerKK2016}]\label{lem:pivotdecom} 
	Let~$D$ be the canonical split decomposition of a connected graph~$G$. 
	If $xy\in E(G)$, then $D\wedge xy$ is the canonical split decomposition of $G\wedge xy$.  
\end{lemma}

An \emph{internal} edge of a tree is an edge that is not incident with a vertex of degree~$1$.  
We also need the following result due to Bouchet~\cite{Bouchet1988}.

\begin{lemma}[Bouchet~\cite{Bouchet1988}]\label{thm:bouchettree}
  A canonical split decomposition of a tree $T$ can be constructed by replacing each internal edge of $T$ by a path of length $3$, the middle edge of the path  being the
  marked edge. 
\end{lemma}

A bag of a split decomposition is a \emph{branching bag} if it is incident with at least three marked edges. 
Let $\mathcal{M}$ be the set of all distance-hereditary graphs in which each connected component admits a canonical split decomposition with the property that every branching bag is a complete bag.
We will prove that ${\cal C}$ is a subclass of ${\cal M}$, so every graph of ${\cal C}$ is a distance-hereditary graph with the additional property that  every branching bag in its canonical split decomposition is a complete bag.
In order to do this we need one more lemma.

\begin{lemma}[Bouchet~\cite{Bouchet1988}]\label{l-bouch}
A connected graph is distance-hereditary if and only if every bag of its canonical split decomposition is either a star or a complete graph.
\end{lemma}

\begin{lemma}\label{l-subset}
$\mathcal{C}$ is a subclass of~$\mathcal{M}$.
\end{lemma}

\begin{proof}
Let $G\in {\cal C}$. 
By the definition of~$\mathcal{C}$, we find that $G$ is obtained from the $1$-subdivision~$T$ of some tree by performing local complementations at vertices of degree at least~$3$.   
By Lemma~\ref{thm:bouchettree}, we can first construct the canonical split decomposition of~$T$.
By Lemma~\ref{lem:localdecom} we can then perform local complementations at corresponding unmarked vertices in the canonical split decomposition of $T$ to obtain the canonical split decomposition of~$G$ (we refer again to \figurename~\ref{fig:decomposition} for an example).
By construction, every bag of the canonical split decomposition of $G$ is a star or a complete bag.
Hence, by Lemma~\ref{l-bouch}, we find that $G$ is a distance-hereditary graph.
From our construction we also note that every branching bag is complete. 
We conclude that $G\in {\cal M}$.
\end{proof}

We will now prove that for every tree $T$ that is not a caterpillar, ${\cal M}$ is $T$-pivot-minor-free.
In order to do this we first show that $\mathcal{M}$ is closed under taking pivot-minors.

\begin{lemma}\label{lem:closed}
The class $\mathcal{M}$ is closed under taking pivot-minors.
\end{lemma}
\begin{proof}
	Let $G\in \mathcal{M}$. 
	It is sufficient to show the following:
	\begin{enumerate}[(1)]
	\item for $v\in V(G)$, $G-v$ is in $\mathcal{M}$, 
	\item for an edge $wz\in E(G)$, $G\pivot wz$ is in $\mathcal{M}$.
	\end{enumerate}
	We may assume that $G$ is connected. Let $D$ be the canonical split decomposition of $G$.
	
	We first show (2). Let $wz$ be an edge of $G$.
	By Lemma~\ref{lem:pivotdecom},
	$D\wedge wz$ is the canonical split decomposition of $G\wedge wz$.
	Let~$P$ be the path in~$D$ linking~$w$ and~$z$ where unmarked edges and marked edges appear alternately.
	By the definition of pivoting in a canonical split decomposition, 
	 we obtain $D\wedge wz$ from $D$ as follows: for each bag~$B$ containing an unmarked edge~$xy$ of~$P$, we replace~$B$ with $B\wedge xy$. 
	 It is easy to observe that if $B$ is a complete bag, then $B\wedge xy$ is again a complete bag, 
	 and if $B$ is a star bag, then $B\wedge xy$ is again a star bag.
	Therefore, $G\wedge wz$ is again contained in $\mathcal{M}$.

	It remains to prove (1). Suppose that $v\in V(G)$ and let $B$ be the bag containing $v$.
	We may assume that $D$ has at least two bags, otherwise the statement follows immediately.

\medskip
\noindent
{\bf Case 1.} $B$ is a complete bag.\\
	If $\abs{V(B)}\ge 4$, then after removing $v$ in $G$,
we find that $B$
	 is still a complete bag of size at least $3$.
	So, $D-v$ is a canonical split decomposition of $G-v$.
	If $\abs{V(B)}=3$, then $B-v$ is merged with one of the neighbour bags of $B$.
	This process does not change the type of the neighbour bag. 
	It is possible that the two neighbour bags $B_1$ and $B_2$ of $B$ in $D$ are star bags, 
	and they can be merged after $B-v$ is merged with a neighbour bag. 
	In this case, each of $B_1$ and $B_2$ has at most two neighbour bags, 
	and after merging $B_1$ and $B_2$, it is again a star bag that has at most two neighbour bags.
	Thus, $G-v$ is in $\mathcal{M}$ again.
	
\medskip
\noindent
{\bf Case 2.} $B$ is a star bag and $v$ is a leaf of $B$.\\
	If $\abs{V(B)}\geq 4$, then $D-v$ is the canonical split decomposition of $G-v$.
	If $\abs{V(B)}=3$, then $B-v$ is merged with one of the neighbour bags of $B$.
	This process does not change the type of the neighbour bag.
	There might be two cases where the two neighbour bags $B_1$ and $B_2$ of $B$ in $D$  
	are merged after $B-v$ is merged with a neighbour bag. 
	If $B_1$ and $B_2$ are complete bags, then the merged bag becomes a complete bag.
	If $B_1$ and $B_2$ are star bags, then each of them has at most two neighbour bags in $D$, 
	and after merging, the new bag is again a star bag that has at most two neighbour bags.
	So, $G-v$ is in $\mathcal{M}$.

\medskip
\noindent
{\bf Case 3.} $B$ is star bag and $v$ is the center of $B$.\\
	Then $v$ is a cut vertex of $G$; that is, $G-v$ is disconnected.
	Furthermore, each component of $G-v$ either consists of a single vertex, or it admits a split decomposition obtained from a connected component of 
	$D-V(B)$ by removing a leaf of a star bag or a vertex in a complete bag. Thus, each component of $G-v$ is in $\mathcal{M}$, 
	and thus, $G-v$ is also in $\mathcal{M}$.
	This concludes the proof of the lemma.
\end{proof}

We also need a known characterization of graphs of linear rank-width at most~$1$ in terms of their canonical split decompositions.

\begin{lemma}[Kant\'e and Kwon~{\cite{KanteK2018}}]\label{thm:charlrw1split}
Let $G$ be a connected graph with canonical split decomposition~$D$.
Then $G$ has linear rank-width at most~$1$  if and only if~$G$ is distance-hereditary and 
the decomposition tree of  $D$ is a path.
\end{lemma}

We are now ready to prove Theorem~\ref{t-main1}.

\medskip
\noindent
{\bf Theorem~\ref{t-main1} (restated).}
{\it If~$T$ is a tree that is not a caterpillar, then the class of $T$-pivot-minor-free distance-hereditary graphs has unbounded linear rank-width.}

\begin{proof}
Let~$T$ be a tree that is not a caterpillar. As ${\cal C}\subseteq {\cal M}$ by Lemma~\ref{l-subset} and ${\cal C}$ has unbounded linear rank-width by Lemma~\ref{l-c}, it follows that ${\cal M}$ has unbounded linear rank-width. Moreover, ${\cal M}$ is a subclass of the class of distance-hereditary graphs.
Hence, to prove the theorem it remains to show that every graph in ${\cal M}$ is $T$-pivot-minor-free. 

Suppose, for contradiction, that~$T$ is a pivot-minor of some graph~$H\in \mathcal{M}$.
As $\mathcal{M}$ is closed under pivot-minors by Lemma~\ref{lem:closed}, we find that
$T\in \mathcal{M}$.
Let $L_T$ be a canonical split decomposition of $T$.
Since $T$ is a tree, $L_T$ has no complete bags.
So, by the definition of $\mathcal{M}$, $L_T$ has no branching bags, and thus, the decomposition tree of $L_T$ is a path.
Since $T$ is distance-hereditary 
and the decomposition tree of $L_T$ is a path, 
$T$ has linear rank-width at most $1$ by Lemma~\ref{thm:charlrw1split}.
Thus, by Lemma~\ref{lem:charlrw1tree}, $T$ is a caterpillar, a contradiction.
\end{proof}

\section{The Proof of Theorem~\ref{thm:pmdh}}\label{pndh}

In this section, we prove Theorem~\ref{thm:pmdh}, which states that $P_n$-pivot-minor-free distance-hereditary graphs have bounded linear rank-width.

To prove Theorem~\ref{thm:pmdh}, we use the canonical split decomposition of a distance-hereditary graph, discussed in Section~\ref{sec:caterpillar}.
A sequence $B_1, B_2, \ldots, B_n$ of distinct bags in a canonical split decomposition is a \emph{path of bags} if for each $i\in \{1, \ldots, n-1\}$, $B_{i+1}$ is a neighbour bag of $B_i$.
As we will explain in the proof of Theorem~\ref{thm:pmdh}, it follows from the definition of a canonical split decomposition, that 
at least half of the bags in a path of bags are star bags.
By applying some pivot operations, 
we can extract a long path as a pivot-minor in this case.
So, we may assume that the decomposition tree of the canonical split decomposition of a given graph has no long path.
We use the following result which relates the path-width of a decomposition tree and the linear rank-width of the graph to conclude the theorem.

\begin{proposition}[Kant\'e and Kwon~{\cite{KanteK2018}}]\label{Kantek2018}
Let $D$ be the canonical split decomposition of a connected distance-hereditary graph $G$, and let~$T_D$ be the decomposition tree of $D$. Then $\frac{1}{2}\pw(T_D)\le \lrw(G) \le \pw(T_D) + 1$.
\end{proposition}

We also use a tight version of Theorem~\ref{t-rs83}. We will use the fact that if a graph contains a minor isomorphic to $P_n$, 
then it also contains a subgraph isomorphic to $P_n$.

\begin{theorem}[Bienstock, Robertson, Seymour, and Thomas~\cite{BienstockRST1991}]\label{t-brst91}
For every tree~$T$ on $n$ vertices, the class of $T$-minor-free graphs has path-width at most $n-2$.
\end{theorem} 

We are now ready to prove Theorem~\ref{thm:pmdh}.

\medskip
\noindent
{\bf Theorem~\ref{thm:pmdh} (restated).}
{\it Let $n\ge 3$ be an integer. Every $P_n$-pivot-minor-free distance-hereditary graph has linear rank-width at most $2n-5$.}

\begin{proof}
Let $G$ be a distance-hereditary graph having no pivot-minor isomorphic to $P_n$.
We will show that $G$ has linear rank-width at most $2n-5$.
We may assume that $G$ is connected.
Let~$D$ be the canonical split decomposition of $G$ and let~$T_D$ be its decomposition tree.

We claim that $D$ has no path with $2n-4$ bags.
Suppose that such a path of bags $B_1, B_2, \ldots, B_{2n-4}$ exists. As no two complete bags are neighbour bags in a canonical split decomposition, 
at most $n-2$ bags in the sequence are complete bags.
Thus, there are at least $n-2$ star bags.
Let $B_{i_1}, B_{i_2}, \ldots, B_{i_t}$ be the sequence of all star bags in $B_1, B_2, \ldots, B_{2n-4}$ where
$1\le i_1<i_2< \cdots <i_t\le 2n-4$.
For convenience, we assign $B_0=B_{2n-3}=\emptyset$.

We claim that for 
every
$k \in \{1,\ldots,t\}$, there is a graph $G_k$ pivot-equivalent to $G$ with a canonical split decomposition $D_k$
such that 
\begin{itemize}
	\item[($\ast$)] for every $k' \in \{1,\ldots,k\}$, the bag $D_k[V(B_{i_{k'}})]$ is a star bag whose center has no neighbour in $V(B_{i_{k'}-1})\cup V(B_{i_{k'}+1})$.
\end{itemize}
For $k=1$, suppose that $D$ does not satisfy the property $(\ast)$. 
Choose a vertex $v'$ in $B_{i_1}$ that has no neighbour in $V(B_{i_{1}-1})\cup V(B_{i_{1}+1})$. Such a vertex exists, as each bag has at least three vertices.
Let $v$ be the vertex of $G$ represented by $v'$ 
(recall that for every marked vertex, there is a vertex of $G$ that
 it represents).
We further choose a vertex $w$ of $G$ represented by the center $w'$ of $B_{i_1}$.
Since $v'w'$ is an edge of $B_{i_1}$, $v$ is linked to $w$ in $D$, and therefore, $v$ is adjacent to $w$ in $G$. 

Note that for a star $H$ with a leaf $b$ and center $a$, $H\pivot ab$ is the star 
with vertex set~$V(H)$ whose center is~$b$.
Thus, in $D\pivot vw$, $v'$ becomes the center of $B_{i_1}\pivot v'w'$, which has no neighbour in $V(B_{i_{1}-1})\cup V(B_{i_{1}+1})$.
Thus, $G_1=G\pivot vw$ and $D_1=D\pivot vw$ satisfy $(\ast)$.

Now, assume that $k>1$ and the property ($\ast$) is satisfied for $k-1$.
If the center of $D_{k-1}[V(B_{i_k})]$ has no neighbour in $V(B_{i_{k}-1})\cup V(B_{i_{k}+1})$, 
then $G_k=G_{k-1}$ and $D_k=D_{k-1}$ satisfy $(\ast)$.
So, we may assume that 
the center of $D_{k-1}[V(B_{i_k})]$ has a neighbour in $V(B_{i_{k}-1})\cup V(B_{i_{k}+1})$.
We distinguish two cases.

\medskip
\noindent
\textbf{Case 1.} The center of $D_{k-1}[V(B_{i_k})]$ has a neighbour in $V(B_{i_{k}-1})$.

Note that if $D_{k-1}[V(B_{i_{k}-1})]$ is a star bag, then by the inductive hypothesis, 
its center has no neighbour in $V(B_{i_k})$. But this is not possible by the definition of a canonical split decomposition.
Thus, $D_{k-1}[V(B_{i_{k}-1})]$ is a complete bag. We choose a vertex $v$ in $G_{k-1}$ represented by a vertex in $D_{k-1}[V(B_{i_{k}-1})]$ having no neighbour in $V(B_{i_{k}-2})\cup V(B_{i_{k}})$
and choose a vertex $w$ in $G_{k-1}$ represented by a vertex in $D_{k-1}[V(B_{i_k})]$ having no neighbour in $V(B_{i_{k}-1})\cup V(B_{i_{k}+1})$.
Observe that $v$ is adjacent to $w$ in $G_{k-1}$, because the center of $D_{k-1}[V(B_{i_k})]$ has a neighbour in $V(B_{i_k-1})$.
Thus, in $D_{k-1}\pivot vw$, the bag induced by $V(B_{i_k})$ is a star bag whose center has no neighbour in $V(B_{i_{k}-1})\cup V(B_{i_{k}+1})$.
As the bags on $V(B_{i_1}), \ldots, V(B_{i_{k-1}})$ are not changed by this pivot operation, $G_k=G_{k-1}\pivot vw$ and $D_k=D_{k-1}\pivot vw$ satisfy $(\ast)$.

\medskip
\noindent
\textbf{Case 2.} The center of $D_{k-1}[V(B_{i_k})]$ has a neighbour in $V(B_{i_{k}+1})$.

We choose a vertex $v$ in $G_{k-1}$ represented by a vertex in $D_{k-1}[V(B_{i_{k}+1})]$ having no neighbour in $V(B_{i_{k}})\cup V(B_{i_{k}+2})$
and choose a vertex $w$ in $G_{k-1}$ represented by a vertex in $D_{k-1}[V(B_{i_k})]$ having no neighbour in $V(B_{i_{k}-1})\cup V(B_{i_{k}+1})$ and linked to $v$ in $D_{k-1}$.
Such a vertex $w$ exists, because  the center of $D_{k-1}[V(B_{i_k})]$ has a neighbour in $V(B_{i_{k}+1})$, 
and thus, if $D_{k-1}[V(B_{i_{k}+1})]$ is a star, then its center has a neighbour in $V(B_{i_k})$.
This implies that $v$ is adjacent to $w$ in $G_{k-1}$.
In $D_{k-1}\pivot vw$, the bag induced by $V(B_{i_k})$ is a star bag whose center has no neighbour in $V(B_{i_{k}-1})\cup V(B_{i_{k}+1})$.
As the bags on $V(B_{i_1}), \ldots, V(B_{i_{k-1}})$ are not changed by this pivot operation, 
$G_k=G_{k-1}\pivot vw$ and $D_k=D_{k-1}\pivot vw$ satisfy $(\ast)$.

\medskip
\noindent
Hence, we have found that the claim holds.

\medskip
\noindent
Now, in $D_t$, let $v_j$ be a vertex of $G_t$ represented by the center of $D_t[V(B_{i_j})]$ for each $j \in \{1,\ldots,t\}$,
and let $v_0$ be a vertex of $G_t$ represented by a leaf of $D_t[V(B_{i_1})]$ which has no neighbour in $V(B_{i_1+1})$, 
and let $v_{t+1}$ be a vertex of $G_t$ represented by a leaf of $D_t[V(B_{i_t})]$ which has no neighbour in $V(B_{i_t-1})$.
It is not difficult to check that $v_0v_1v_2 \cdots v_tv_{t+1}$ is an induced path of $G_t$ on $t+2\ge n$ vertices.
This contradicts the assumption that $G$ has no pivot-minor isomorphic to $P_n$.
We conclude that $D$ has no path with $2n-4$ bags $B_1, B_2, \ldots, B_{2n-4}$.

The above means that the decomposition tree $T_D$ has no path on $2n-4$ vertices. By Theorem~\ref{t-brst91},  
$T_D$ has path-width at most $2n-6$. By Proposition~\ref{Kantek2018}, $G$ has linear rank-width at most $2n-5$.
\end{proof}

\section{The Proof of Theorem~\ref{t-maintwosubgraphs}}\label{claw}

In this section, we prove Theorem~\ref{t-maintwosubgraphs}, which states that the class of $(K_3,S_{1,2,2})$-free graphs has linear rank-width at most~$58$. We prove the following statements in this order:

\begin{enumerate}
\item bipartite $2P_2$-free graphs, which form a subclass of bipartite $(P_1+2P_2)$-free graphs, have linear rank-width at most~$1$;
\item bipartite $(P_1+2P_2)$-free graphs, which form a subclass of bipartite $S_{1,2,2}$-free graphs, have linear rank-width at most~$3$;
\item bipartite $S_{1,2,2}$-free graphs have linear rank-width at most~$3$;
\item non-bipartite $(K_3,C_5,S_{1,2,2})$-free graphs have linear rank-width at most~$3$; and
\item  $(K_3,S_{1,2,2})$-free graphs with an induced $C_5$ have linear rank-width at most~$58$.
\end{enumerate}

Note that Statements 3--5 cover all cases for proving Theorem~\ref{t-maintwosubgraphs}.	
So, we first consider bipartite $2P_2$-free graphs.

	\begin{lemma}\label{lem:bipartitechain}
	Every bipartite $2P_2$-free graph has linear rank-width at most~$1$.
	\end{lemma}
	\begin{proof}
	Let~$G$ be a bipartite $2P_2$-free graph with bipartition~$(A,B)$.
It is well known~\cite{Yannakakis1982} that a bipartite graph with bipartition $(X_1,X_2)$ is $2P_2$-free if and only if it is a bipartite chain graph, that is, for each $i\in \{1,2\}$, the neighbourhoods of the vertices in~$X_i$ can be ordered linearly with respect to the inclusion relation. 
We may assume that $G$ is connected.
Hence, as $G$ is $2P_2$-free, we can define a sequence
	$A_1, A_2, \ldots, A_m$ of pairwise vertex-disjoint subsets of~$A$ such that
	\begin{itemize}
	\item $A_1\cup A_2\cup \cdots \cup A_m=A$, 	
	\item each~$A_i$ is a maximal set of pairwise twins in~$G$, 
	\item for integers $i,j\in \{1, \ldots, m\}$ with $i<j$, $N_G(A_i)\subsetneq N_G(A_j)$.
	\end{itemize}
	If $m=1$, then $G$ is complete bipartite. In this case, we take a linear ordering $L_1$ of $A_1$ and a linear ordering $L_2$ of $V(G)\setminus A_1$ arbitrarily. 
	It is not hard to see that $L_1\oplus L_2$ is a linear ordering of width at most $1$. Hence $G$ has linear rank-width at most~$1$.
	
	Now suppose that $m\geq 2$. In this case, $G$ is not complete bipartite.
	Notice that for each $i\in \{2, \ldots, m\}$, there is a vertex~$v\in B$ that has a neighbour in~$A_i$ but does not have a neighbour in $A_1\cup \cdots \cup A_{i-1}$; otherwise, vertices in $A_{i-1}\cup A_i$ have the same neighbourhood in~$B$, which contradicts the maximality of~$A_i$.	
	For each $i\in \{2, \ldots, m\}$, let $B_i:=N_G(A_i)\setminus N_G(A_{i-1})$, and let $B_1:=N_G(A_1)$.
	Since~$G$ is connected, we have $B=B_1\cup B_2\cup \cdots \cup B_m$.
	
	For each $i\in \{1, \ldots, m\}$, let~$L^A_i$ be 
	an ordering of~$A_i$ and~$L^B_i$ be an ordering of~$B_i$.
	It is not difficult to check that the linear ordering
	$L^B_1\oplus L^A_1\oplus L^B_2 \oplus L^A_2\oplus \cdots \oplus L^B_m\oplus L^A_m$ has width at most~$1$.
	\end{proof}

We now consider bipartite $(P_1+2P_2)$-free graphs and show the following lemma.
        
\begin{lemma}\label{lem:bipartitep5free}
Every bipartite $(P_1+\nobreak 2P_2)$-free graph has linear rank-width at most~$3$.	
\end{lemma}
\begin{proof}
	Let~$G$ be a bipartite $(P_1+\nobreak 2P_2)$-free graph with bipartition~$(A,B)$.
	Then $G\times (A,B)$ is $P_5$-free. We may assume without loss of generality that $G\times (A,B)$ is connected. Then, as $G\times (A,B)$ is also bipartite, $G\times (A,B)$ is readily seen to be $2P_2$-free.
	By Lemma~\ref{lem:bipartitechain}, $G\times (A,B)$ has linear rank-width at most~$1$.
	By Lemma~\ref{lem:bipartitecomplement}, $G=G\times (A,B)\times (A,B)$ has linear rank-width at most~$3$.
\end{proof}

We now consider $S_{1,2,2}$-free bipartite graphs and need two results by Lozin~\cite{Lozin2000}.

\begin{lemma}[Lozin~\cite{Lozin2000}]\label{lem:lozin1}
Every connected bipartite $(S_{1,2,2},P_7)$-free graph is $(P_1+\nobreak 2P_2)$-free. 
\end{lemma}

\begin{lemma}[Lozin~\cite{Lozin2000}]\label{lem:lozin2}
Let~$G$ be a bipartite $S_{1,2,2}$-free graph with no twins. If $G$ contains an induced~$P_7$,
then~$G$ is $K_{1,3}$-free 
(and thus $G$ has maximum degree at most~$2$).
\end{lemma}

\begin{proposition}\label{t-clawpmfree1}
Every bipartite $S_{1,2,2}$-free graph has linear rank-width at most~$3$.
\end{proposition}
\begin{proof}
Let~$G$ be a bipartite $S_{1,2,2}$-free graph with bipartition~$(A,B)$.
We may assume that~$G$ is connected.
If~$G$ is $P_7$-free, then by Lemma~\ref{lem:lozin1}, $G$ is $(P_1+\nobreak 2P_2)$-free, and by Lemma~\ref{lem:bipartitep5free}, $G$ has linear rank-width at most~$3$.
Thus, we may assume that~$G$ contains an induced subgraph isomorphic to~$P_7$. Note that $G$ is not a complete bipartite graph.

Let $I_1, I_2, \ldots, I_m$ be the vertex partition of~$G$ such that each~$I_i$ is a maximal set of pairwise twins in~$G$.
Since~$G$ is connected and bipartite, each~$I_i$ is contained in one of $A$ or~$B$.
Let $G_1:=G//I_1 //I_2 // \cdots //I_m$.
Note that~$G_1$ is also connected.
We claim that~$G_1$ has no twins. Note that $G_1$ is not an edge, because $G$ is not a complete bipartite graph. So, $G_1$ has at least three vertices.
Moreover, $G_1$ still has an induced subgraph isomorphic to $P_7$, as $P_7$ has no twins.

Suppose for contradiction that~$G_1$ has two twins~$v_1$ and~$v_2$, and assume that~$v_1$ and~$v_2$ were identified from~$I_{i_1}$ and~$I_{i_2}$ for some~$i_1$ and~$i_2$, respectively.
Since each~$v_i$ has a neighbour and $G_1$ has at least three vertices, $v_1$ and~$v_2$ are in the same part of the bipartition, and thus~$I_{i_1}$ and~$I_{i_2}$ are contained in the same part of the bipartition of~$G$.
Thus~$I_{i_1}$ and~$I_{i_2}$ have the same neighbourhoods in~$G$, contradicting the fact that they are maximal sets of pairwise twins in~$G$.
So, $G_1$ has no twins.
Then $G_1$ has linear rank-width at most~$2$ because by Lemma~\ref{lem:lozin2} every vertex has degree at most $2$.
By Lemma~\ref{lem:twinreduced}, we conclude that $G$ has linear rank-width at most~$3$.
\end{proof}

We now consider $(K_3,C_5,S_{1,2,3})$-free graphs and need the following result as a lemma.

\begin{lemma}[Dabrowski, Dross, and Paulusma~\cite{DabrowskiDP2016}]\label{lem:c5free}
Let $G$ be a connected $(K_3,C_5,S_{1,2,3})$-free graph that does not contain a pair of false twins.
Then $G$ is either bipartite or an induced cycle.
\end{lemma}

	\begin{proposition}\label{prop:clawpmfree2}
	Every non-bipartite $(K_3, C_5, S_{1,2,2})$-free graph has linear rank-width at most~$3$.	
	\end{proposition}
	\begin{proof}
	Let~$G$ be a connected non-bipartite $(K_3, C_5, S_{1,2,2})$-free graph.  
	By Lemma~\ref{lem:c5free}, 
	$G$ is a graph obtained from an induced cycle $C=c_1c_2 \cdots c_kc_1$ by adding false twins.
	 For each $i\in \{1, \ldots, k\}$, let~$U_i$ be the maximal set of false twins containing~$c_i$ in~$G$. 
	 As $G//U_1 //U_2// \cdots //U_k$ is isomorphic to $C$ and $C$ has linear rank-width at most $2$,
	 by Lemma~\ref{lem:twinreduced}, we find that $G$ has linear rank-width at most $3$.
	\end{proof}	

We now consider  $(K_3, S_{1,2,2})$-free graphs that contain an induced~$C_5$. We first introduce some additional terminology and lemmas. 
A graph is \emph{$3$-partite} if its vertex set can be partitioned into three independent sets.
We need the following known result.

\begin{theorem}[Alecu et al.~\cite{ALZW18}]\label{t-alzw18}
Let~$G$ be a $3$-partite graph on $n$ vertices with vertex partition $(V_1, V_2, V_3)$ such that 
\begin{itemize}
\item[(a)] for every $a\in V_1$, $b\in V_2$, $c\in V_3$, $G[a,b,c]$ is isomorphic to neither~$K_3$ nor~$3P_1$,
\item[(b)] $G[V_1\cup V_2]$, $G[V_1\cup V_3]$, and $G[V_2\cup V_3]$ are $2P_2$-free.
\end{itemize}
Then $V(G)$ admits a linear ordering $x_1,\ldots,x_n$
  and a labelling $\ell:V(G)\to \{a,b,c\}$ such that $x_ix_j\in E(G)$ for $i<j$ if and only if $(\ell(x_i),\ell(x_j))\in \{(a,b),(b,c),(c,a)\}$.
\end{theorem}

The following lemma follows from Theorem~\ref{t-alzw18} after observing that each cut of the linear ordering $x_1,\ldots,x_n$ has cut-rank at most $3$, because it has at most three different rows. 
	
\begin{lemma}\label{lem:3partitegraph}
Let~$G$ be a $3$-partite graph with vertex partition $(V_1, V_2, V_3)$ such that 
\begin{itemize}
\item[(a)] for every $a\in V_1$, $b\in V_2$, $c\in V_3$, $G[a,b,c]$ is isomorphic to neither~$K_3$ nor~$3P_1$,
\item[(b)] $G[V_1\cup V_2]$, $G[V_1\cup V_3]$, and $G[V_2\cup V_3]$ are $2P_2$-free.
\end{itemize}
Then~$G$ has linear rank-width at most~$3$. 
\end{lemma}

Let $G$ be a graph and $V_1, V_2, V_3$ be 
three pairwise
disjoint independent sets of $G$. We denote the subgraph of $G$ induced by $V_1\cup V_2\cup V_3$ as $G[V_1, V_2, V_3]$.
Moreover, if $G[V_1, V_2, V_3]$ satisfies conditions~(a) and~(b) in Lemma~\ref{lem:3partitegraph}, then we say that $G[V_1, V_2, V_3]$ is \emph{nice}.

We are now ready to prove the following result. We note that Brandst\"adt, Mahfud and Mosca~\cite{BrandstadtMM16} gave an alternative proof of the result  from~\cite{DabrowskiDP2016}, which shows that $(K_3,S_{1,2,2})$-free graphs have bounded rank-width. Some parts of the proof of our result below are similar to parts of the proof of~\cite{BrandstadtMM16}. As we need to use slightly different arguments, we have chosen to keep our proof self-contained. However, we explicitly indicate whenever there is overlap between our arguments and the ones used in~\cite{BrandstadtMM16}.

	\begin{proposition}\label{p-c5}
	Every $(K_3, S_{1,2,2})$-free graph that contains an induced~$C_5$ has linear rank-width at most~$58$.	
	\end{proposition}
	
        \begin{proof}
	Let~$G$ be a $(K_3, S_{1,2,2})$-free graph that contains an induced subgraph $C$ isomorphic to~$C_5$.
	We may assume without loss of generality that~$G$ is connected. 
We write $C=c_1c_2c_3c_4c_5c_1$ and interpret subscripts modulo~$5$.
	Let $U:=V(G)\setminus V(C)$.
	
	Since~$G$ is $K_3$-free, every vertex in~$U$ has either no neighbours in~$C$ or exactly one neighbour in~$C$ or exactly two neighbours, which are not consecutive in~$C$.
	
We claim that every vertex of $U$ has a neighbour in $C$.
This can be seen as follows. Suppose, for contradiction, that $U$ contains a vertex that has no neighbour in $C$. As $G$ is connected, this means that $U$ contains two vertices
 $w_1$ and $w_2$, such that $w_1$ has a neighbour in $C$ and $w_2$ is adjacent to $w_1$ but $w_2$ has no neighbour in $C$.
	If $w_1$ has exactly one neighbour $c_i$ in $C$, then $G[c_i, c_{i+1},c_{i-1}, c_{i-2},  w_1, w_2]$ is isomorphic to $S_{1,2,2}$.
	If $w_1$ has two neighbours $c_{i-1}, c_{i+1}$ in $C$, then $G[w_1, w_2, c_{i-1}, c_i, c_{i-2}, c_{i-3}]$ is isomorphic to $S_{1,2,2}$. 	 
However, both cases are not possible, as	 $G$ is $S_{1,2,2}$-free. Hence, we conclude that every vertex in $U$ has a neighbour in $C$.
	
	Consequently, we can partition~$U$ into ten parts $\{V_1, \ldots, V_5, W_1, \ldots, W_5\}$ such that for each $i\in \{1, \ldots ,5\}$,
	\begin{itemize}
	\item $V_i$ is the set of vertices whose unique neighbour in~$C$ is~$c_i$, and
	\item $W_i$ is the set of vertices that are adjacent to~$c_{i-1}$ and~$c_{i+1}$.
	\end{itemize}
	Each set in $\{V_1, \ldots, V_5, W_1, \ldots, W_5\}$ is an independent set as~$G$ is $K_3$-free.
	We verify basic relations between these parts. 
	Let $i\in \{1,\ldots, 5\}$. 
	\begin{enumerate}
	\item[(1)] $V_i$ is complete to $V_{i-1}\cup V_{i+1}$, and anti-complete to $V_{i-2}\cup V_{i+2}$.
	
	Suppose, for contradiction, that there are $a\in V_i$ and $b\in V_{i+1}$ that are not adjacent. 
	Then $G[c_i, a, c_{i+1}, b, c_{i-1}, c_{i-2}]$ is isomorphic to $S_{1,2,2}$, a contradiction. 
	Thus, there are no such vertices. This implies that $V_i$ is complete to $V_{i+1}$, and by symmetry also complete to $V_{i-1}$.
	Suppose that there are $a\in V_i$ and $b\in V_{i+2}$ that are adjacent.
	Then $G[c_i, c_{i+1}, a, b, c_{i-1}, c_{i-2}]$ is isomorphic to $S_{1,2,2}$, a contradiction.
	Hence, $V_i$ is anti-complete to $V_{i+2}$ and by symmetry also anti-complete to $V_{i-2}$.\dia
	
	\item[(2)] $W_i$ is anti-complete to $W_{i-2}\cup W_{i+2}$.

	This is because $G$ is $K_3$-free.\dia
	\item[(3)] $W_i$ is complete to~$V_i$, and anti-complete to $V_{i-1}\cup V_{i+1}$.

	Suppose that there are vertices $a\in W_i$ and $b\in V_i$ that are not adjacent to each other. 
	Then $G[c_{i-1}, a, c_i, b, c_{i-2}, c_{i-3}]$ is isomorphic to $S_{1,2,2}$, a contradiction.
	This implies that $W_i$ is complete to $V_i$.
	As $G$ is $K_3$-free, $W_i$ is anti-complete to $V_{i-1}\cup V_{i+1}$.	\dia
	
	\item[(4)] $G[V_i\cup W_{i+2}]$, $G[V_i\cup W_{i-2}]$, $G[W_i\cup W_{i+1}]$ are $2P_2$-free.
	
	Suppose that there are $a_1, a_2\in V_i$ and $b_1, b_2\in W_{i+2}$ such that $a_1b_1, a_2b_2\in E(G)$ and $a_1b_2, a_2b_1\notin E(G)$.
	Then $G[c_i, c_{i-1}, a_1, b_1, a_2, b_2]$ is isomorphic to $S_{1,2,2}$, a contradiction.
So, $G[V_i\cup W_{i+2}]$ is $2P_2$-free, and by symmetry $G[V_i\cup W_{i-2}]$ is $2P_2$-free. 
Suppose that there are $a_1, a_2\in W_i$ and $b_1, b_2\in W_{i+1}$ such that $a_1b_1, a_2b_2\in E(G)$ and $a_1b_2, a_2b_1\notin E(G)$.
	Then $G[c_{i-1}, c_{i-2}, a_1, b_1, a_2, b_2]$ is isomorphic to $S_{1,2,2}$, a contradiction. \dia
	
	\item[(5)] For $v\in V_i, w\in W_{i+2}, z\in W_{i-2}$, $\{v,w,z\}$ is not an independent set.
	
	Suppose that such $v,w,z$ forming an independent set exist. 
	Then $G[c_i, v, c_{i+1}, w, c_{i-1}, z]$ is isomorphic to $S_{1,2,2}$, a contradiction. \dia
      \end{enumerate}

      By Claims~(4) and~(5)     
and the fact that $G$ is $K_3$-free, 
we deduce that for each $i \in \{1,\ldots,5\}$,
      $G[V_i,W_{i-2},W_{i+2}]$ is a nice $3$-partite graph. 
      Moreover, $V_i$ is complete to
      $V_{i-1}, V_{i+1}, W_i$ and anti-complete to $V_{i+2}, V_{i+3}, W_{i-1},W_{i+1}$, whereas~$W_i$ is anti-complete to $W_{i-2},W_{i+2}$.
      By doing bipartite complementations between $V_i$ and $V_{i+1}\cup W_i$ for each $i \in \{1,\ldots,5\}$,
      we may assume that each edge
not incident to a vertex of~$C$
      belongs to $G[V_i,W_{i+2},W_{i-2}]$ for some $i \in \{1,\ldots,5\}$.
      
  Our goal is now to compute a refined set of
      pairwise non-intersecting $3$-partite graphs by doing a small number of  bipartite complementations such that each new $3$-partite graph satisfies conditions~(a) and~(b) of Lemma \ref{lem:3partitegraph}.

      We first observe that each $G[V_i,W_{i-2},W_{i+2}]$ intersects only $G[V_{i+1},W_{i-1},W_{i-2}]$ and $G[V_{i-1}, W_{i+2},W_{i+1}]$.  We now aim, for each $i \in \{1,\ldots,5\}$, to split $V_i$, $W_{i-2}$ and $W_{i+2}$ in such a way that we can construct the desired non-intersecting $3$-partite graphs after some bipartite
      complementations. We will use the same construction as in \cite{BrandstadtMM16}, but the way we use the different sets differs from \cite{BrandstadtMM16}. 

      For each $i \in \{1,\ldots,5\}$, let us define the following partition of $W_{i}$:
      \begin{align*}
        W_{i}^- &= \{x\in W_{i}\mid x\ \textrm{has a non-neighbour in}\  W_{i+1}\},\\
        W_{i}^+ &=\{x\in W_{i}\setminus W_{i}^- \mid x\ \textrm{has a non-neighbour in}\ W_{i-1}\},\\
        W_{i}^* &= W_{i}\setminus (W_{i}^-\cup W_{i}^+).
      \end{align*}
      By definition, $W_i^+\cup W_i^*$ is complete to $W_{i+1}$ and $W_i^-$ is complete to $W_{i+1}^*$.
      We claim that $W_i^-$ is also complete to $W_{i+1}^-$.
      Suppose that a vertex $x\in W_i^-$ has a non-neighbour $y$
          in $W_{i+1}^-$. Then, by definition $y$ has a non-neighbour $z$ in $W_{i+2}$. Therefore, $G[c_{i+1},z,x,c_{i-1},c_{i+2},y]$ is isomorphic to $S_{1,2,2}$ because $x$
          is not adjacent to $z$ by (2), a contradiction.
	Thus, $W_i^-$ is complete to $W_{i+1}^-$.
	
      We now show some relationships between the $V_i$'s and $W_i$'s, which were also proven in \cite{BrandstadtMM16}, but we add the proofs for completeness.  Let $i \in \{1,\ldots,5\}$. 
      \begin{enumerate}[(a)]          
        \item $V_i$ is anti-complete to $W_{i+2}^+$ and $W_{i-2}^-$.

 Suppose, for contradiction, that a vertex $x\in V_i$ has a neighbour $z\in W_{i+2}^+\cup W_{i-2}^-$.
First suppose that $z\in W_{i+2}^+$.
Then, by definition, $z$ has a non-neighbour $y\in W_{i+1}$. However, now $G[c_i,c_{i-1},
          x,z,y,c_{i+2}]$ is isomorphic to $S_{1,2,2}$, a contradiction.  Now suppose that $z\in W_{i-2} ^-$. Then, by definition, $z$ has a non-neighbour $y\in W_{i-1}$. However, now
          $G[c_i,c_{i+1},y,c_{i-2},x,z]$ is
          isomorphic to $S_{1,2,2}$, another contradiction. \dia
       
        \item If a vertex $x\in V_i$ has a neighbour in $W_{i+2}^*$ (resp. $W_{i-2}^*$), then $x$ is complete to $W_{i+2}^-$ and anti-complete to $W_{i-2}$ (resp.
          complete to $W_{i-2}^+$ and anti-complete to $W_{i+2}$).

          Suppose $x\in V_i$ has a neighbour $y$ in $W_{i+2}^*$. Because $W_{i+2}^*$ is complete to $W_{i+3}=W_{i-2}$, we find that
          $x$ is anti-complete to $W_{i-2}$. Suppose that $x$ has
          a non-neighbour $z$ in $W_{i+2}^-$. By definition, there is a vertex $w\in W_{i-2}$ not adjacent to $z$, but adjacent to $y$. Then, $G[y,x,w,c_{i-1},c_{i+1},z]$
          is isomorphic to $S_{1,2,2}$, a contradiction.

          Now suppose that $x\in V_i$ has a neighbour $y$ in $W_{i-2}^*$. Because by definition $W_{i-2}^*$ is complete to $W_{i-3}=W_{i+2}$, we find that $x$ is anti-complete to
          $W_{i+2}$. Suppose that $x$ has a non-neighbour $z$ in $W_{i-2}^+$.  By definition, $z$ has a non-neighbour $w$ in $W_{i+2}$, which is adjacent to $y$. Again,
          $G[y,x,w,c_{i+1},c_{i-1},z]$ is isomorphic to $S_{1,2,2}$, a contradiction. \dia
        \end{enumerate}

\noindent
  	For each $i \in \{1,\ldots,5\}$, let
        \begin{align*}
          V_i^1 & =  \{x\in V_i\mid x\ \textrm{has a neighbour in $W_{i+2}^*$}\},\\
          V_i^2 &= \{x\in V_i\mid x\ \textrm{has a neighbour in $W_{i-2}^*$}\},\\
          V_i^3 &= \{x\in V_i\mid x\ \textrm{has no neighbour in}\ W_{i+2}^* \cup W_{i-2}^*\}.
        \end{align*}
          From (a), we know that for each $x\in V_i$,
in the graph $G[V_i \cup W_{i-2} \cup W_{i+2}]$,
$x$ can have neighbours in only $W_{i+2}^- \cup W_{i+2}^*$ and in $W_{i-2}^+\cup W_{i-2}^*$.  
          From (b), every vertex in $V_i^1$ is complete to $W_{i+2}^-$ and anti-complete to $W_{i-2}$, 
          and every vertex in $V_i^2$ is complete to $W_{i-2}^+$ and anti-complete to $W_{i+2}$.
          Thus, $V_i^1, V_i^2, V_i^3$ are disjoint sets.
               
               The common neighbours between a vertex $x\in V_i$ and $y\in V_{i+1}$ are in $W_{i-2}^*$ and the common neighbours between $x\in
        V_i$ and $y\in V_{i-1}$ are in $W_{i+2}^*$.  
For each $i \in \{1,\ldots,5\}$, we define the following $3$-partite graphs: 
        \begin{align*}
          G_i^1 & = G[V_i^1, V_{i-1}^2, W_{i+2}^*],\\
          G_i^2 &= G[V_i^3,W_{i+2}^-,W_{i-2}^+].
        \end{align*}

	Now, we do some bipartite complementations.
       First we do bipartite complementations between
        $W_i^+\cup W_i^*$ and $W_{i+1}$, and between $W_i^-$ and $W_{i+1}^-\cup W_{i+1}^*$ for each $i$. 
        The resulting graph has the property that the edges between $W_i$ and $W_{i+1}$ are always between $W_i^-$ and $W_{i+1}^+$. 
	Secondly, we do bipartite complementations between $V_i^1$ and $W_{i+2}^-$, and $V_i^2$ and $W_{i-2}^+$ for each $i$. 
	Lastly, we remove $C$.
	Observe that the edges of the remaining graph are all contained in one $G_i^j$ for some $i \in \{1,\ldots,5\}$
        and $j \in \{1,2\}$, and by definition the $G_i^j$'s are pairwise disjoint. 

        Because all of the graphs $G_i^j$ are disjoint, the linear rank-width of the resulting graph is the maximum linear rank-width of its connected components. Since each connected component is a nice
        $3$-partite graph, 
        by Lemma \ref{lem:3partitegraph}, we conclude that the linear rank-width of the resulting graph is at most $3$. We will now count the number of
        times we applied bipartite complementations. For each $i \in \{1,\ldots,5\}$, we did one bipartite complementation to keep only the edges between $V_i$ and $W_{i-2}\cup W_{i+2}$, 
        and then two bipartite
        complementations to remove the edges between~$W_i$ and~$W_{i+1}$ apart from those between $W_i^-$ and $W_{i+1}^+$, and finally two bipartite complementations to remove the edges between $V_i^1$ and $W_{i+2}^-$, and $V_i^2$ and $W_{i-2}^+$, that is, in total five bipartite complementations, resulting in a total of $25$ bipartite complementations. So, by Lemma \ref{lem:bipartitecomplement}, the graph $G-V(C)$ has linear rank-width at most
        $3+2*25=53$. Because $|V(C)|=5$, we conclude that the linear rank-width of $G$ is at most $58$. 
	\end{proof}
	
We are now ready to prove Theorem~\ref{t-maintwosubgraphs}.

\medskip
\noindent
{\bf Theorem~\ref{t-maintwosubgraphs} (restated).}
{\it Every $(K_3, S_{1,2,2})$-free graph has linear rank-width at most~$58$.}

\begin{proof} 
Let $G$ be a $(K_3, S_{1,2,2})$-free graph. We may assume that~$G$ is connected. If $G$ is bipartite, we use Proposition~\ref{t-clawpmfree1}.
If $G$ is non-bipartite but $C_5$-free, we use Proposition~\ref{prop:clawpmfree2}. In the remaining case, we use Proposition~\ref{p-c5}.
\end{proof}

\section{Concluding Remarks}\label{sec:conclusion}

In this paper we researched the relationship between pivot-minors and boundedness of linear rank-width. We first proved that for every tree~$T$ that is not a caterpillar, the class of $T$-pivot-minor-free graphs has unbounded linear rank-width.
We then posed Conjecture~\ref{q-1}, which states that an affirmative answer can be found whenever $T$ {\it is} a caterpillar.
We were only able to give an affirmative answer to this conjecture that holds for every caterpillar $T$, if the class of $T$-pivot-minor-free graphs is, in addition, also distance-hereditary.
We also proved that the class of $K_{1,3}$-pivot-minor-free graphs has bounded linear rank-width.
As a next step for proving Conjecture~\ref{q-1}, it seems natural to consider the case where $T=K_{1,r}$ for $r\geq 4$.
We also proved Conjecture~\ref{q-1} for $P=P_4$.
Since Conjecture~\ref{q-1} is equivalent to the alternative conjecture that for every path $P$, the class of $P$-pivot-minor-free graphs has bounded linear rank-width, the case where $P=P_5$ is another interesting open case.

For obtaining our results (in particular, the case where $T=K_{1,3}$) we followed a general strategy consisting of two steps.  We believe this strategy is also useful for making further progress towards Conjecture~\ref{q-1}. However, Step~1 of the strategy requires us to find a hereditary graph class (class of graphs that can be characterized by a set of forbidden induced subgraphs) that contains the class of $T$-pivot-minor-free graphs under consideration. In general, finding an appropriate hereditary graph class is a challenging task.

The fact that $K_{1,3}$-pivot-minor-free graphs have bounded linear rank-width follows from a stronger result that we showed, namely that $(K_3,S_{1,2,2})$-free graphs have bounded linear rank-width. Showing this stronger result will be useful for a systematic study on the boundedness of linear rank-width of  $(H_1,H_2)$-free graphs.
Such a classification  already exists for $H$-free graphs, as observed in~\cite{DJP19}: for a graph $H$, the class of $H$-free graphs has bounded linear rank-width if and only if~$H$ is a subgraph of $P_3$ not isomorphic to $3P_1$.
We note that similar classifications also exist for other width parameters: for the tree-width of $(H_1,H_2)$-free graphs~\cite{BBJPPV20}, 
which was later generalized to a classification for  tree-width of ${\cal H}$-free graphs, where~${\cal H}$ is a finite set of graphs~\cite{LR},
rank-width of $H$-free graphs
(see~\cite{DP16}), rank-width of $H$-free bipartite graphs~\cite{DP14,Lo02,LV08}, and up to five non-equivalent open cases, rank-width of $(H_1,H_2)$-free graphs (see~\cite{BDJP19} or~\cite{DJP19}), and for the mim-width of $H$-free graphs~\cite{BHMPP20}, whereas there is still an infinite number of open cases left for the mim-width of $(H_1,H_2)$-free graphs~\cite{BHMPP20}.

We leave a systematic study into boundedness of linear rank-width of $(H_1,H_2)$-free graphs for future research. Here, we only collect known results. The class of $(H_1, H_2)$-free graphs has bounded linear rank-width if 
\begin{itemize}
	\item one of $H_1$ and $H_2$ is a subgraph of $P_3$ that is not isomorphic to $3P_1$~\cite{DJP19},
	\item $(H_1, H_2)=(K_3, S_{1,2,2})$ and $(3P_1, \overline{S_{1,2,2}})$ (Theorem~\ref{t-maintwosubgraphs})
	\item $(H_1, H_2)=(P_4, F)$ where $F$ is a threshold graph~\cite{BKV17}.
\end{itemize}
The class of $(H_1, H_2)$-free graphs has unbounded linear rank-width if 
\begin{itemize}
	\item $(H_1, H_2)=(K_3, S_{1,2,3})$ or $(3P_1, \overline{S_{1,2,3}})$~\cite{ALZ18}, 
	\item $(H_1, H_2)=(P_4, F)$ where $F$ is not a threshold graph~\cite{BKV17}, 
	\item all known cases where $(H_1, H_2)$-free graphs have unbounded rank-width.
\end{itemize}

\bibliographystyle{plain}

\begin{thebibliography}{10}

\bibitem{AFP13}
Isolde Adler, Arthur~M. Farley, and Andrzej Proskurowski.
\newblock Obstructions for linear rank-width at most 1.
\newblock {\em Discrete Appl. Math.}, 168:3--13, 2014.

\bibitem{AdlerK2015}
Isolde Adler and Mamadou~Moustapha Kant{\'e}.
\newblock Linear rank-width and linear clique-width of trees.
\newblock {\em Theoret. Comput. Sci.}, 589:87--98, 2015.

\bibitem{AdlerKK2016}
Isolde Adler, Mamadou~Moustapha Kant\'{e}, and O-joung Kwon.
\newblock Linear rank-width of distance-hereditary graphs {I}. {A}
  polynomial-time algorithm.
\newblock {\em Algorithmica}, 78(1):342--377, 2017.

\bibitem{ALZ18}
Bogdan Alecu, Mamadou~Moustapha Kant\'e, Vadim~V. Lozin, and Viktor Zamaraev.
\newblock Between clique-width and linear clique-width of bipartite graphs.
\newblock {\em Discrete Mathematics}, 343(8):111926, 2020.

\bibitem{ALZW18}
Bogdan Alecu, Vadim~V. Lozin, Dominique de~Werra, and Viktor Zamaraev.
\newblock Letter graphs and geometric grid classes of permutations:
  characterization and recognition.
\newblock {\em Discrete Appl. Math.}, 283:482--494, 2020.

\bibitem{BienstockRST1991}
Dan Bienstock, Neil Robertson, Paul Seymour, and Robin Thomas.
\newblock Quickly excluding a forest.
\newblock {\em J. Combin. Theory Ser. B}, 52(2):274--283, 1991.

\bibitem{BBJPPV20}
Hans~L. Bodlaender, Nick Brettell, Matthew Johnson, Giacomo Paesani, Dani\"el
  Paulusma, and Erik~Jan van Leeuwen.
\newblock Steiner trees for hereditary graph classes: {A} treewidth
  perspective.
\newblock {\em Theoret. Comput. Sci.}, 867:30--39, 2021.

\bibitem{BDJP19}
Marthe Bonamy, Nicolas Bousquet, Konrad~K. Dabrowski, Matthew Johnson, Dani\"el
  Paulusma, and Th\'eo Pierron.
\newblock Graph isomorphism for {$(H_1,H_2)$-free} graphs: an almost complete
  dichotomy.
\newblock {\em Algorithmica}, 83(3):822--852, 2021.

\bibitem{Bouchet1988}
Andr{\'e} Bouchet.
\newblock Transforming trees by successive local complementations.
\newblock {\em J. Graph Theory}, 12(2):195--207, 1988.

\bibitem{Bouchet1989}
Andr{\'e} Bouchet.
\newblock Connectivity of isotropic systems.
\newblock In {\em Combinatorial {M}athematics: {P}roceedings of the {T}hird
  {I}nternational {C}onference ({N}ew {Y}ork, 1985)}, volume 555 of {\em Ann.
  New York Acad. Sci.}, pages 81--93. New York Acad. Sci., New York, 1989.

\bibitem{Bouchet1994}
Andr{\'e} Bouchet.
\newblock Circle graph obstructions.
\newblock {\em J. Combin. Theory Ser. B}, 60(1):107--144, 1994.

\bibitem{BrandstadtMM16}
Andreas Brandstädt, Suhail Mahfud, and Raffaele Mosca.
\newblock Bounded clique-width of $({S}_{1,2,2}$, triangle$)$-free graphs.
\newblock arXiv:1608.01820, 2016.

\bibitem{BHMPP20}
Nick Brettell, Jake Horsfield, Andrea Munaro, Giacomo Paesani, and Dani\"el
  Paulusma.
\newblock Bounding the mim-width of hereditary graph classes.
\newblock {\em J. Graph Theory}, to appear.

\bibitem{BKV17}
Robert Brignall, Nicholas Korpelainen, and Vincent Vatter.
\newblock Linear clique-width for hereditary classes of cographs.
\newblock {\em Journal of Graph Theory}, 84(4):501--511, 2017.

\bibitem{Co14}
Bruno Courcelle.
\newblock Clique-width and edge contraction.
\newblock {\em Inform. Process. Lett.}, 114(1-2):42--44, 2014.

\bibitem{CourcelleMR00}
Bruno Courcelle, Johann~A. Makowsky, and Udi Rotics.
\newblock Linear time solvable optimization problems on graphs of bounded
  clique-width.
\newblock {\em Theory Comput. Syst.}, 33(2):125--150, 2000.

\bibitem{CO00}
Bruno Courcelle and Stephan Olariu.
\newblock Upper bounds to the clique width of graphs.
\newblock {\em Discrete Appl. Math.}, 101(1-3):77--114, 2000.

\bibitem{CunninghamE80}
William~H. Cunningham and Jack Edmonds.
\newblock A combinatorial decomposition theory.
\newblock {\em Canad. J. Math.}, 32(3):734--765, 1980.

\bibitem{DabrowskiDP2016}
Konrad~K. Dabrowski, Fran\c{c}ois Dross, and Dani\"{e}l Paulusma.
\newblock Colouring diamond-free graphs.
\newblock {\em J. Comput. System Sci.}, 89:410--431, 2017.

\bibitem{DDJKKOP18}
Konrad~K. Dabrowski, Fran{\c{c}}ois Dross, Jisu Jeong, Mamadou~Moustapha
  Kant{\'{e}}, O{-}joung Kwon, Sang{-}il Oum, and Dani{\"{e}}l Paulusma.
\newblock Computing small pivot-minors.
\newblock {\em Proc. {W}{G} 2018, {L}{N}{C}{S}}, 11159:125--138, 2018.

\bibitem{DDJKKOP19}
Konrad~K. Dabrowski, Fran{\c{c}}ois Dross, Jisu Jeong, Mamadou~Moustapha
  Kant{\'{e}}, O{-}joung Kwon, Sang{-}il Oum, and Dani{\"{e}}l Paulusma.
\newblock Tree pivot-minors and linear rank-width.
\newblock {\em Proc. EuroComb 2019, Acta Mathematica Universitatis Comenianae},
  88(3):577--583, 2019.

\bibitem{DJP19}
Konrad~K. Dabrowski, Matthew Johnson, and Dani\"el Paulusma.
\newblock Clique-width for hereditary graph classes.
\newblock {\em London Mathematical Society Lecture Note Series}, 456:1--56,
  2019.

\bibitem{DP14}
Konrad~K. Dabrowski and Dani\"{e}l Paulusma.
\newblock Classifying the clique-width of {$H$}-free bipartite graphs.
\newblock {\em Discrete Appl. Math.}, 200:43--51, 2016.

\bibitem{DP16}
Konrad~K. Dabrowski and Dani\"el Paulusma.
\newblock Clique-width of graph classes defined by two forbidden induced
  subgraphs.
\newblock {\em The Computer Journal}, 59(5):650--666, 2016.

\bibitem{EGW01}
Wolfgang Espelage, Frank Gurski, and Egon Wanke.
\newblock How to solve {NP-hard} graph problems on clique-width bounded graphs
  in polynomial time.
\newblock {\em Proc. WG 2001, LNCS}, 2204:117--128, 2001.

\bibitem{Ga11}
Robert Ganian.
\newblock Thread graphs, linear rank-width and their algorithmic applications.
\newblock {\em Proc. IWOCA 2010, LNCS}, 6460:38--42, 2011.

\bibitem{GKMW2019}
Jim Geelen, O{-}joung Kwon, Rose McCarty, and Paul Wollan.
\newblock The grid theorem for vertex-minors.
\newblock {\em Journal of Combinatorial Theory, Series B}, in press.

\bibitem{GK03}
Michael~U. Gerber and Daniel Kobler.
\newblock Algorithms for vertex-partitioning problems on graphs with fixed
  clique-width.
\newblock {\em Theoretical Computer Science}, 299(1):719--734, 2003.

\bibitem{JKO16}
Jisu Jeong, Eun~Jung Kim, and Sang{-}il Oum.
\newblock The ``art of trellis decoding'' is fixed-parameter tractable.
\newblock {\em IEEE Trans. Inform. Theory}, 63(11):7178--7205, 2017.

\bibitem{KanteK2018}
Mamadou~Moustapha Kant\'{e} and O{-}joung Kwon.
\newblock Linear rank-width of distance-hereditary graphs {II}. {V}ertex-minor
  obstructions.
\newblock {\em European J. Combin.}, 74:110--139, 2018.

\bibitem{Ka08}
Navin Kashyap.
\newblock Matroid pathwidth and code trellis complexity.
\newblock {\em SIAM J. Discrete Math.}, 22(1):256--272, 2008.

\bibitem{KR03}
Daniel Kobler and Udi Rotics.
\newblock Edge dominating set and colorings on graphs with fixed clique-width.
\newblock {\em Discrete Applied Mathematics}, 126(2--3):197--221, 2003.

\bibitem{KMOW19}
O{-}joung Kwon, Rose McCarty, Sang{-}il Oum, and Paul Wollan.
\newblock Obstructions for bounded shrub-depth and rank-depth.
\newblock {\em J. Combin. Theory Ser. B}, 149:76--91, 2021.

\bibitem{ko14}
O-joung Kwon and Sang{-}il Oum.
\newblock Graphs of small rank-width are pivot-minors of graphs of small
  tree-width.
\newblock {\em Discrete Appl. Math.}, 168:108--118, 2014.

\bibitem{Lozin2000}
Vadim~V. Lozin.
\newblock {$E$}-free bipartite graphs.
\newblock {\em Diskretn. Anal. Issled. Oper. Ser. 1}, 7(1):49--66, 103, 2000.

\bibitem{Lo02}
Vadim~V. Lozin.
\newblock Bipartite graphs without a skew star.
\newblock {\em Discrete Mathematics}, 257(1):83--100, 2002.

\bibitem{LR}
Vadim~V. Lozin and Igor Razgon.
\newblock Tree-width dichotomy.
\newblock arXiv:2012.01115, 2020.

\bibitem{LV08}
Vadim~V. Lozin and Jordan Volz.
\newblock The clique-width of bipartite graphs in monogenic classes.
\newblock {\em International Journal of Foundations of Computer Science},
  19(2):477--494, 2008.

\bibitem{Nesetril19}
Jaroslav Ne\v{s}et\v{r}il, Patrice Ossona~de Mendez, Roman Rabinovich, and
  Sebastian Siebertz.
\newblock Classes of graphs with low complexity: {T}he case of classes with
  bounded linear rankwidth.
\newblock {\em European J. Combin.}, 91:103223, 2021.

\bibitem{Ou05}
Sang{-}il Oum.
\newblock Rank-width and vertex-minors.
\newblock {\em J. Combin. Theory Ser. B}, 95(1):79--100, 2005.

\bibitem{Oum2009}
Sang{-}il Oum.
\newblock Excluding a bipartite circle graph from line graphs.
\newblock {\em J. Graph Theory}, 60(3):183--203, 2009.

\bibitem{Ou17}
Sang{-}il Oum.
\newblock Rank-width: algorithmic and structural results.
\newblock {\em Discrete Appl. Math.}, 231:15--24, 2017.

\bibitem{OS06}
Sang{-}il Oum and Paul Seymour.
\newblock Approximating clique-width and branch-width.
\newblock {\em J. Combin. Theory Ser. B}, 96(4):514--528, 2006.

\bibitem{Ra07}
Micha\"el Rao.
\newblock {MSOL} partitioning problems on graphs of bounded treewidth and
  clique-width.
\newblock {\em Theoretical Computer Science}, 377(1--3):260--267, 2007.

\bibitem{GMI}
Neil Robertson and P.~D. Seymour.
\newblock Graph minors. {I}. {E}xcluding a forest.
\newblock {\em J. Combin. Theory Ser. B}, 35(1):39--61, 1983.

\bibitem{GMV}
Neil Robertson and P.~D. Seymour.
\newblock Graph minors. {V}. {E}xcluding a planar graph.
\newblock {\em J. Combin. Theory Ser. B}, 41(1):92--114, 1986.

\bibitem{GMXX}
Neil Robertson and P.~D. Seymour.
\newblock Graph minors. {XX}. {W}agner's conjecture.
\newblock {\em J. Combin. Theory Ser. B}, 92(2):325--357, 2004.

\bibitem{Takahashi1994}
Atsushi Takahashi, Shuichi Ueno, and Yoji Kajitani.
\newblock Minimal acyclic forbidden minors for the family of graphs with
  bounded path-width.
\newblock {\em Discrete Math.}, 127(1-3):293--304, 1994.

\bibitem{Yannakakis1982}
Mihalis Yannakakis.
\newblock The complexity of the partial order dimension problem.
\newblock {\em SIAM J. Algebraic Discrete Methods}, 3(3):351--358, 1982.

\end{thebibliography}

\end{document}